\documentclass{ruthesis}
\usepackage{amsmath,amsfonts,amsthm,amssymb,caption,subcaption,enumerate,latexsym,hyperref,url,graphicx}

\usepackage{tikz,color,lscape}

\newtheorem{theorem}{Theorem}[chapter]

\newtheorem{definition}[theorem]{Definition}

\newtheorem{lemma}[theorem]{Lemma}
\newtheorem{proposition}[theorem]{Proposition}

\makeatletter
\newtheorem*{rep@theorem}{\rep@title}
\newcommand{\newreptheorem}[2]{%
\newenvironment{rep#1}[1]{%
 \def\rep@title{#2 \ref{##1}}%
 \begin{rep@theorem}}%
 {\end{rep@theorem}}}
\makeatother

\newreptheorem{theorem}{Theorem}
\newreptheorem{lemma}{Lemma}
\newcommand{\RR}{\mathbb{R}}      % for Real numbers
\begin{document}
\phd
\title{An Experimental Mathematics Approach to Several Combinatorial Problems}
\author{YUKUN YAO}
\program{Mathematics}
\director{Doron Zeilberger}
\approvals{4} % CHANGE THIS IF YOU HAVE EXTRA COMMITTEE MEMBERS
\submissionyear{2020} %CHANGE TO APPLICABLE YEAR...
\submissionmonth{May} % ... AND MONTH

\abstract{

Experimental mathematics is an experimental approach to mathematics in which programming and symbolic computation are used to investigate mathematical objects, identify properties and patterns, discover facts and formulas and even automatically prove theorems. 

With an experimental mathematics approach, this dissertation deals with several combinatorial problems and demonstrates the methodology of experimental mathematics.

We start with parking functions and their moments of certain statistics. Then we discuss about spanning trees and ``almost diagonal" matrices to illustrate the methodology of experimental mathematics. We also apply experimental mathematics to Quicksort algorithms to study the running time. Finally we talk about the interesting peaceable queens problem.

}
\beforepreface

%%%%%%%%%%%%%%%%%%%
\acknowledgements{

First and foremost, I would like to thank my advisor, Doron Zeilberger, for all of his help, guidance and encouragement throughout my mathematical adventure in graduate school at Rutgers. He introduced me to the field of experimental mathematics and many interesting topics in combinatorics. Without him, this dissertation would not be possible.

I am grateful to many professors at Rutgers: Michael Kiessling, for serving on my oral qualifying exam and thesis defense committee; Vladimir Retakh, for serving on my defense committee; Shubhangi Saraf and Swastik Kopparty, for teaching me combinatorics and serving on my oral exam committee.

I am also grateful to Neil Sloane, for introducing me to the peaceable queens problem and numerous amazing integer sequences and for serving on my defense committee. 

I would like to thank other combinatorics graduate students here at Rutgers. From them I learned about a lot of topics in this rich and fascinating area. 

I would like to express my appreciation to my officemate, Lun Zhang, for interesting conversation and useful remarks. 

I appreciate Shalosh B. Ekhad's impeccable computing support and the administrative support from graduate director Lev Borisov and graduate secretary Kathleen Guarino.

Finally, I thank my parents. They always support and encourage me to pursue what I want in all aspects of my life.

}

%%%%%%%%%%%%%%%%%%%

\figurespage     % (make a page for list of figures)
%\tablespage      % (make a page for list of tables)

\afterpreface 

%%%%%%%%%%%%%%%%%%%%%%%%%%%%%%%%%%%%%%%%%%%%%%%%%%%%%%%%%%%%%%%%%%%%%%%%
\chapter{Introduction} 
\label{ch:introduction}
%%%%%%%%%%%%%%%%%%%%%%%%%%%%%%%%%%%%%%%%%%%%%%%%%%%%%%%%%%%%%%%%%%%%%%%%
Since the creation of computers, they have been playing a more and more important role in our everyday life and the advancement of science and technology, bringing efficiency and convenience and reshaping our world. 

Especially, the use of computers is becoming increasingly popular in mathematics. The proof of the four color theorem would be impossible without computers. Compared with human beings, computers are faster, more powerful, tireless, less error-prone. Computers can do much more than numerical computation. With the development of computer science and symbolic computation, experimental mathematics, as an area of mathematics, has been growing fast in the last several decades.

With experimental mathematics, it is much more efficient and easier to look for a pattern, test a conjecture, utilize data to make a discovery, etc. Computers can be programmed to make conjectures and provide rigorous proofs with little or no human intervention. They can also do what humans cannot do or what takes too long to complete, e.g., solving a large linear system, analyzing a large data set, symbolically computing complicated recurrence relations. Just as machine learning revolutionizes computer science, statistics and information technology, experimental mathematics revolutionizes mathematics.

The main theme of the dissertation is to use the methods of experimental mathematics to study different problems, and likewise, to illustrate the methodology and power of experimental mathematics by showing some case studies and how experimental mathematics works under various situations. 

In Chapter 2, we discuss the first problem that is related to the area statistic of of parking functions. Our methods are purely finitistic and elementary, taking full advantage, of course, of our beloved silicon servants. We first introduce the background and definition of parking functions and their generalizations. For $a$-parking functions, we derive the recurrence relation and the number of them when the length is $n$. Furthermore, a bijection between $a$-parking functions and labelled rooted forests is discovered (or possibly re-discovered). Then we consider the sum and area statistics. With the weighted counting of these statistics, the explicit formula between expectation and higher moments can be found. We also look at the limiting distribution of the area statistic, which is Airy distribution.

In Chapter 3, we use two instructive case studies on spanning trees of grid graphs and ``almost diagonal'' matrices, to show that often, just like Alexander the Great before us, the simple, ``cheating'' solution
to a hard problem is the best. So before you spend days (and possibly years) trying to answer
a mathematical question by analyzing and trying to `understand' its structure, let your computer generate enough data, and then let it guess the answer.
Often its guess can be  proved by a quick `hand-waving' (yet fully rigorous) `meta-argument'.

In Chapter 4, we apply experimental mathematics to algorithm analysis. Using recurrence relations, combined with symbolic computations, we make a detailed study of the running times of numerous variants of the celebrated Quicksort algorithms, where we consider the variants of single-pivot and multi-pivot Quicksort algorithms as discrete probability problems. With nonlinear difference equations, recurrence relations and experimental mathematics techniques, explicit expressions for expectations, variances and even higher moments of their numbers of comparisons and swaps can be obtained. For some variants, Monte Carlo experiments are performed, the numerical results are demonstrated and the scaled limiting distribution is also discussed.

In Chapter 5, finally, we discuss, and make partial progress on, the peaceable queens problem, the protagonist of OEIS sequence {\tt A250000}. Symbolically, we prove that Jubin's construction of two pentagons is at least a local optimum. Numerically, we find the exact numerical optimums for some specific configurations. Our method can be easily applied to more complicated configurations with more parameters. 

All accompanying Maple packages and additional input/output files can be found at the author's homepage:

{\tt http://sites.math.rutgers.edu/\~{}yao}  .

The accompanying Maple package for Chapter 2 is {\tt ParkingStatistics.txt}. There are lots of output files and nice pictures on the front of this chapter.

The accompanying Maple packages for Chapter 3 are {\tt JointConductance.txt}, {\tt GFMa-} {\tt trix.txt} and {\tt SpanningTrees.txt}. There are also numerous sample input and output files on the front of this chapter.

The accompanying Maple packages for Chapter 4 are {\tt QuickSort.txt} and {\tt Findrec.} {\tt txt}. {\tt QuickSort.txt} is the main package of this chapter and all procedures mentioned in the chapter are from this package unless noted otherwise. {\tt Findrec.txt} is mainly used to find a recurrence relation, i.e., difference equation of moments from the empirical data.

The accompanying Maple package for Chapter 5 is {\tt PeaceableQueens.txt}. There are lots of output files and nice pictures on the front of this chapter as well.

Experimental mathematics has made a huge impact on mathematics itself and how mathematicians discover new mathematics so far and is the mathematics of tomorrow. In the information era, the skyscraper of mathematics is becoming taller and taller. Hence we need tools better than pure human logic to maintain and continue building this skyscraper. While the computing capacity, patience and time of humans are limited, experimental mathematics, and ultimately, automated theorem proving, will be the choice of history.

%%%%%%%%%%%%%%
\chapter{The Statistics of Parking Functions}
%%%%%%%%%%%
This chapter is adapted from \cite{47}, which has been published on {\it The Mathematical Intelligencer}. It is also available on arXiv.org, number 1806.02680.

\section{Introduction}

Once upon a time, way back in the nineteen-sixties, there was a one-way street (with no passing allowed), with $n$  parking spaces
bordering the sidewalk. Entering the street were  $n$ cars, each driven by a loyal {\it husband}, and sitting next to him,
dozing off, was his capricious (and a little bossy) {\it wife}. At a random time (while still along the street), the wife wakes
up and orders her husband, {\bf park here, darling!}. If that space is unoccupied, the hubby gladly obliges,
and if the parking space is occupied, he  parks, if possible, at the next still-empty parking space.
Alas, if all the latter parking spaces are occupied, he has to go around the block, and drive back to the
beginning of this one-way street, and then look for the first available spot. Due to construction, this
wastes half an hour, making the wife very cranky. 

{\bf Q}: What is the probability that no one has to go around the block?

{\bf A}: $(n+1)^{n-1}/n^n \, \asymp  \, \frac{e}{n+1}$.

Both the question and its elegant answer are due to Alan Konheim and Benji Weiss \cite{29}.

Suppose wife $i$ ($1 \leq i \leq n$) prefers parking space $p_i$, then the preferences of the wives can
be summarized as an array $(p_1, \dots, p_n)$, where $1 \leq p_i \leq n$. So altogether there are $n^n$ possible
preference-vectors, starting from $(1, \dots , 1)$ where it is clearly possible for everyone  to park,
and ending with $(n,...,n)$ (all $n$), where every wife prefers the last parking space, and of course
it is impossible. Given a preference vector $(p_1, \dots, p_n)$, let $(p_{(1)}, \dots, p_{(n)})$
be its {\it sorted} version, arranged in (weakly) increasing order. \hfill\break
For example if $(p_1,p_2,p_3,p_4)=(3,1,1,4)$ then  $(p_{(1)},p_{(2)},p_{(3)},p_{(4)})=(1,1,3,4)$.

We invite our readers to convince themselves that a parking space preference vector  $(p_1, \dots, p_n)$
makes it possible for every husband to park without inconveniencing his wife if and only if $p_{(i)} \leq i$ for
$1 \leq i \leq n$. This naturally leads to the following definition.

\begin{definition}[Parking Function]
A vector of positive integers $(p_1, \dots, p_n)$ with $1 \leq p_i \leq n$ is
a {\bf parking function} if its (non-decreasing) sorted version $(p_{(1)}, \dots, p_{(n)})$ 
(i.e. $p_{(1)} \leq p_{(2)} \leq \dots  \leq p_{(n)}$, and the latter is a permutation of the former)
satisfies
$$
p_{(i)} \leq i , \quad (1 \leq i \leq n) .
$$
\end{definition}

As we have already mentioned above, Alan Konheim and Benji Weiss \cite{29} were the first to state and  prove the following theorem.

\begin{theorem}[The Parking Function Enumeration Theorem]
There are $(n+1)^{n-1}$ parking functions of length $n$.
\end{theorem}

There are many proofs of this lovely theorem, possibly the slickest is due to the brilliant human
Henry Pollak, (who apparently did not deem it worthy of publication. It is quoted, e.g. in \cite{15}).
It is nicely described on pp. 4-5 of \cite{41} (see also \cite{42}), hence we will not repeat it here. 
Instead, as a warm-up to the `statistical' part, and to illustrate the power of experiments, we will
give a much uglier proof, that, however, is {\it motivated}.

Before going on to present {\it our} (very possibly not new) `humble' proof, we should mention that one natural way to prove the
Konheim-Weiss theorem is by a {\it bijection} with labeled trees on $n+1$ vertices, that Arthur
Cayley famously proved is also  enumerated by $(n+1)^{n-1}$. The first such bijection, as far as we know,
was given by the great formal linguist, Marco Sch\"utzenberger \cite{37}. This was followed by an elegant bijection
by the {\it classical} combinatorial giants Dominique Foata and John Riordan \cite{15}, and others.

Since we know (at least!) $16$ different proofs of Cayley's formula (see, e.g. \cite{53}), and at least four  different bijections
between parking functions and labeled trees, there
are at least $64$ different proofs (see also \cite{44}, ex. 5.49) of the Parking Enumeration theorem. To these one must add proofs
like Pollak's, and a few other ones.

Curiously, our `new' proof has some resemblance to the very first one in \cite{29}, since they both
use {\it recurrences} (one of the greatest tools in the experimental mathematician's tool kit!),
but our proof is (i) motivated and (ii) experimental (yet fully rigorous).

\section{An Experimental Mathematics Motivated Proof}

When encountering a new combinatorial family, the first task is to write a computer program to
enumerate as many terms as possible, and hope to {\it conjecture} a nice formula.
One can also try and ``cheat" and use the great OEIS, to see whether anyone came up with this sequence
before, and see whether this new combinatorial family is mentioned there.

A very brute force approach, that will not go very far (but would suffice to get the first five terms needed for the OEIS)
is to list the {\it superset}, in this case all the $n^n$ vectors in $\{1 \dots n\}^n$ and for each of them sort it,
and see whether the condition $p_{(i)} \leq i$ holds for all $1 \leq i \leq n$. Then count the vectors that
pass this test.

But a much better way is to use {\bf dynamical programming} to express the desired sequence, let's call it $a(n)$,
in terms of values $a(i)$ for $i<n$. 

Let's analyze the anatomy of a typical parking function of length $n$.
A natural parameter is the number of $1$'s that show up, let's call it $k$ ($0 \leq k \leq n$).
i.e.
$$
p_{(1)}=1 ,\quad \dots , \quad  p_{(k)}=1  , \quad 2 \leq p_{(k+1)} \leq k+1 
 , \quad \dots  , \quad  p_{(n)} \leq n  .
$$
Removing the $1$'s yields a shorter weakly-increasing vector
$$
2 \leq p_{(k+1)}  \leq p_{(k+2)} \leq  \dots \quad  \leq \, p_{(n)} ,
$$
satisfying  
$$
p_{(k+1)} \leq k+1, \quad p_{(k+2)} \leq k+2  , \quad \dots  , \quad  p_{(n)} \leq n  .
$$
Define
$$
(q_1, \dots, q_{n-k})\, := \, (p_{(k+1)}-1, \dots , p_{(n)}-1 )  .
$$
The vector $(q_1, \dots, q_{n-k})$ satisfies
$$
1 \leq q_1 \leq \dots \leq q_{n-k}   ,
$$
and
$$
q_1 \leq k , \quad q_2 \leq k+1  , \quad \dots, \quad q_{n-k} \leq n-1  .
$$

We see that the set of parking functions with exactly $k$ $1$'s may be obtained by taking the above set of vectors of length $n-k$,
adding $1$ to each component, scrambling it in every which way, and inserting the $k$ $1$'s in every which way.

Alas, the `scrambling' of the set of such $q$-vectors is not of the original form. We are forced to consider a
more general object, namely scramblings of vectors of the form $p_{(1)} \leq \dots \leq  p_{(n)}$ with
the condition
$$
p_{(1)} \leq a , \quad p_{(2)} \leq a+1  , \quad \dots  , \quad p_{(n)} \leq a+n-1  ,
$$
for a general, positive integer $a$, not just for $a=1$.
So in order to get the dynamical programming recurrence rolling we are forced to introduce a more general object,
called an {\bf $a$-parking function}. This leads to the following definition.

\begin{definition}[$a$-Parking Function]
 A vector of positive integers $(p_1, \dots, p_n)$ with $1 \leq p_i \leq n+a-1$ is
an {\bf $a$-parking function} if its (non-decreasing) sorted version $(p_{(1)}, \dots, p_{(n)})$ 
(i.e. $p_{(1)} \leq p_{(2)} \leq \dots  \leq p_{(n)}$, and the latter is a permutation of the former)
satisfies
$$
p_{(i)} \leq a+i-1 , \quad (1 \leq i \leq n)  .
$$
\end{definition}

Note that the usual parking functions are the special case $a=1$. So if we would be able to find an efficient
recurrence for counting $a$-parking functions, we would be able to answer our original question.

So let's redo the above `anatomy' for these more general creatures, and hope that the two parameters $n$ and $a$ would
suffice to establish a {\bf recursive scheme}, and we won't need to introduce yet more general creatures.

Let's analyze the anatomy of a typical $a$-parking function of length $n$.
Again, a natural parameter is the number of $1$'s that show up, let's call it $k$ ($0 \leq k \leq n$).
i.e.
$$
p_{(1)}=1 ,\quad \dots  , \quad  p_{(k)}=1  , \quad  2 \leq p_{(k+1)} \leq a+k  , \quad \dots,  p_{(n)} \leq a+n-1   .
$$
Removing the $1$'s yields a sorted vector
$$
2 \leq p_{(k+1)}  \leq p_{(k+2)} \leq  \dots \, \leq \, p_{(n)}  ,
$$
satisfying  
$$
p_{(k+1)} \leq k+a  , \quad p_{(k+2)} \leq k+a+1 , \quad  \dots , \quad  p_{(n)} \leq n+a-1  .
$$
Define
$$
(q_1, \dots, q_{n-k})\, := \, (p_{(k+1)}-1  , \quad \dots  , \quad p_{(n)}-1 )  .
$$
The vector $(q_1, \dots, q_{n-k})$ satisfies
$$
q_1 \leq \dots \leq q_{n-k}  
$$
and
$$
q_1 \leq k+a-1  , \quad q_2 \leq k+a  , \quad \dots , \quad q_{n-k} \leq n+a-1  .
$$

We see that the set of $a$-parking functions with exactly $k$ $1$'s may be obtained by taking the above set of vectors of length $n-k$,
adding $1$ to each component, scrambling it in every which way, and inserting the $k$ $1$'s in every which way.

But now the set of scramblings of the vectors $(q_1, \dots, q_{n-k})$ is an {\bf old friend!}. It is the set of $(a+k-1)$-parking functions of
length $n-k$. To get all $a$-parking functions of length $n$ with exactly $k$ ones we need to take each and every
member of the set of $(a+k-1)$-parking functions of length $n-k$, add $1$ to each component, and insert $k$  ones in every which
way. There are ${{n} \choose {k}}$ ways of doing it. Hence the number of $a$-parking functions of length $n$ with exactly $k$ ones
is  ${{n} \choose {k}}$ times the number of $(a+k-1)$-parking functions of length $n-k$.
Summing over all $k$ between $0$ and $n$ we get the following recurrence.

\begin{proposition}[Fundamental Recurrence for $a$-parking functions]

Let $p(n,a)$ be the number of $a$-parking functions of length $n$. We have the recurrence
$$
p(n,a) \, = \, \sum_{k=0}^{n} \, {{n} \choose {k}} p(n-k,a+k-1)  ,
$$
subject to the boundary conditions $p(n,0)=0$ for $n \geq 1$, and $p(0,a)=1$ for $a \geq 0$.
\end{proposition}

Note that in the sense of Wilf \cite{46}, this already answers the enumeration problem to compute $p(n,a)$ and hence
$p(n,1)=p(n)$, since this gives us a polynomial time algorithm to compute $p(n)$ (and $p(n,a)$).

Moving the term $k=0$ from the right to the left, and denoting $p(n,a)$ by $p_n(a)$ we have
$$
p_n(a)- p_n(a-1) \, = \, \sum_{k=1}^{n} \, {{n} \choose {k}} p_{n-k}(a+k-1) .
$$

Hence we can express $p_n(a)$ as follows, in terms of $p_{m}(a)$ with $m<n$.
$$
p_n(a) = \sum_{b=0}^{a} \left ( \sum_{k=1}^{n} \, {{n} \choose {k}} p_{n-k}(b+k-1) \right ).
$$

Here is the  Maple code that implements it

{\obeylines
{\tt
p:=proc(n,a) local k,b:
if n=0 then  
RETURN(1) 
else 
factor(subs(b=a,sum(expand(add(binomial(n,k)*subs(a=a+k-1,p(n-k,a)),
k=1..n)),a=1..b))): 
fi:
end:
}
}

If you copy-and-paste this onto a Maple session, as well as the line below,

{\tt [seq(p(i,a),i=1..8)];}

you would {\tt immediately} get
$$
[a,a \left( a+2 \right) ,a \left( a+3 \right) ^{2},a \left( a+4 \right) ^{3},a \left( a+5 \right) ^{4},a \left( a+6 \right) ^{5},a \left( a+7 \right) ^{6},a
 \left( a+8 \right) ^{7}] .
$$

Note that these are {\it rigorously} proved exact expressions, in terms of {\it general} $a$ (i.e. {\it symbolic} $a$) for $p_n(a)$, for
$1 \leq n \leq 10$, and we can easily get more. The following {\bf guess} immediately comes to mind
$$
p(n,a)=p_n(a)= a(a+n)^{n-1}  .    
$$

How to prove this rigorously? If you set $q(n,a):=a(a+n)^{n-1}$, since $q(n,0)=0$ and $q(0,a)=1$, the fact that
$p(n,a)=q(n,a)$ would follow {\bf by induction} once you prove that $q(n,a)$ also satisfies the same fundamental recurrence.

$$
q(n,a) \, = \, \sum_{k=0}^{n} \, {{n} \choose {k}} q(n-k,a+k-1) .
$$
In other words, in order to prove that $p(n,a)=a(n+a)^{n-1}$, we have to prove the {\bf identity}
$$
a(a+n)^{n-1} \, = \, \sum_{k=0}^{n} \, {{n} \choose {k}} (a+k-1)(a+n-1)^{n-k-1}  .
$$

\begin{proof}
Let's define 
$$
f(x)\, := \, \sum_{k=0}^{n} \, {{n} \choose {k}} (a+k-1)x^{n-k-1} ,
$$
hence
\begin{equation*} 
\begin{split}
f(x) & = \frac{a-1}{x} \sum_{k=0}^{n} \, {{n} \choose {k}} x^{n-k} + \sum_{k=0}^{n} \, k {{n} \choose {k}} x^{n-k-1} \\
 & = \frac{a-1}{x} \sum_{k=0}^{n} \, {{n} \choose {k}} x^{n-k} + n \sum_{k=0}^{n} \, {{n} \choose {k}} x^{n-k-1} - \sum_{k=0}^{n} \, (n-k){{n} \choose {k}} x^{n-k-1} \\
 & = \frac{a-1+n}{x} \sum_{k=0}^{n} \, {{n} \choose {k}} x^{n-k} - \sum_{k=0}^{n} \, (n-k){{n} \choose {k}} x^{n-k-1}.
\end{split}
\end{equation*}

As an immediate consequence of the {\bf binomial theorem}:
$$
\sum_{k=0}^{n} \, {{n} \choose {k}} x^{n-k} = (1+x)^n
$$
and
$$
 \sum_{k=0}^{n} \,(n-k) {{n} \choose {k}} x^{n-k-1} = n(1+x)^{n-1},
$$
which is trivial to both humans and machines, we have
$$
f(x) = \frac{a-1+x}{x} (1+x)^n - n(1+x)^{n-1}.
$$
By setting $x = a+n-1$, we get
\begin{equation*} 
\begin{split}
f(x) & = (a+n)^n - n(a+n)^{n-1}\\
 & = a(a+n)^{n-1} .
\end{split}
\end{equation*}
This completes the proof.
\end{proof}

We have just rigorously reproved the following well-known theorem.

\begin{theorem}
The number of $a$-parking functions of length $n$ is
$$
p(n,a)=a\,(a+n)^{n-1}  .
$$
In particular, by substituting $a=1$, we reproved the original Konheim-Weiss theorem that $p(n,1)=(n+1)^{n-1}$.
\end{theorem}

\section{Bijection between $a$-Parking Functions \& Labelled Rooted Forests}
We consider forests with $a$ components and totally $a+n$ vertices where the roots in the $a$ components are $1, 2, \dots, a$. Vertices which are not roots are labelled $a+1, a+2, \dots, a+n$.

Let $t(n,a)$ be the number of such labelled rooted forests with $a$ components and $a+n$ vertices. 

\begin{theorem}
The number of labelled rooted forests with $a$ components and $a+n$ vertices is
$$
t(n,a) = a(a+n)^{n-1}.
$$
\end{theorem}

\begin{proof}
When $n=0$, obviously $t(n,a)=1$ for any $a$. When $n \geq 1$ and $a=0$, $t(n,a)=0$ since there does not exist such a tree with zero component and a positive number of vertices.

Since we want to prove $t(n,a) = p(n,a)$, the number of $a$-parking functions of length $n$ and they satisfy the same boundary condition, it is natural to think about the recurrence relation for $t(n,a)$. Consider the number of neighbors of vertex 1, say, the number is $k, (0 \leq k \leq n)$, then remove them with their own subtrees as new components and delete vertex 1. Then there are $a+k-1$ components and $n-k$ non-rooted vertices. Though in this case the labeling of vertices does not follow our rule, there is a unique relabeling which makes it do. When the number of neighbors of vertex 1 is $k$, there are ${{n} \choose {k}}$ choices, so
$$
t(n,a) =  \sum_{k=0}^{n} {{n} \choose {k}} t(n-k, a+k-1) .
$$
It has exactly the same recurrence relation as $p(n,a)$, hence
$$
t(n,a) = p(n,a) = a(a+n)^{n-1}.
$$
\end{proof}

As $p(n,a)$ and $t(n,a)$ are the same, it would be interesting to find some meaningful bijection between $a$-parking functions of length $n$ and labelled rooted forests with $a$ components and $a+n$ vertices. We discover or possibly re-discover a bijection between them. This bijection can be best demonstrated via an example as follows.  

Assume we already have a 2-parking function with length 7, say 5842121, we'd like to map it to a labelled rooted forests with 2 components and 9 vertices where 1 and 2 are the roots for the components. Because the vertices 1 and 2 are already roots, we use the following two-line notation (\#):
\begin{align*}
vertices: && 3 && 4 && 5 && 6 && 7 && 8 && 9 \\
2-parking function: && 5 && 8 && 4 && 2 && 1 && 2 && 1
\end{align*}

Let's consider the weakly-increasing version (*) first where we just sort the second line of (\#):
\begin{align*}
vertices: && 3 && 4 && 5 && 6 && 7 && 8 && 9 \\
2-parking function: && 1 && 1 && 2 && 2 && 4 && 5 && 8
\end{align*}

We interpret (*) as follows: the parent of vertices 3 and 4 is 1, 5's and 6's parent is 2, etc. Hence we have the following forest.
\begin{figure}[h!]
  \center
  \includegraphics[width=0.75\textwidth]{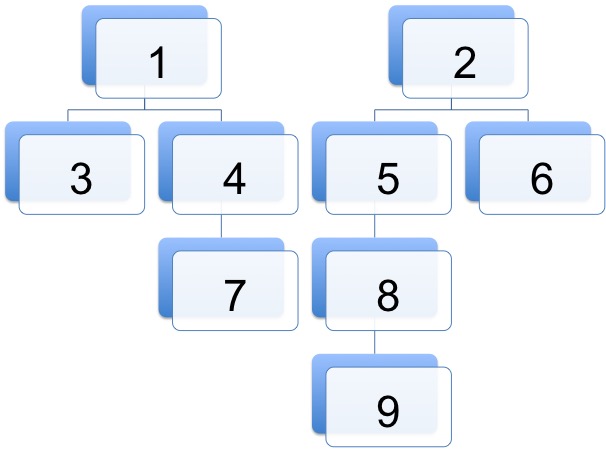}
  \caption{The Forest Associated with (*)}
\end{figure}

If we sort both lines of (\#) according to the second line, then we have
\begin{align*}
vertices: && 7 && 9 && 6 && 8 && 5 && 3 && 4 \\
2-parking function: && 1 && 1 && 2 && 2 && 4 && 5 && 8
\end{align*}

Comparing the first line with that of (*), we have a map
\begin{align*}
3 && 4 && 5 && 6 && 7 && 8 && 9 && \\
\downarrow && \downarrow && \downarrow && \downarrow && \downarrow && \downarrow && \downarrow \\
 7 && 9 && 6 && 8 && 5 && 3 && 4 
\end{align*}

So the 2-parking function 5842121 is mapped to the following forest:
\begin{figure}[h!]
  \center
  \includegraphics[width=0.75\textwidth]{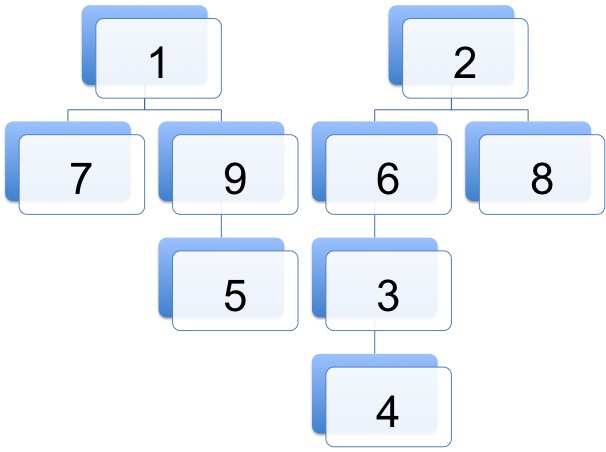}
  \caption{The Labelled Rooted Forest Which the 2-Parking Function 5842121 is mapped to}
\end{figure}
One convention is that when we draw the forests, for the same parent, we always place its children in an increasing order (from left to right). 

Conversely, if we already have the forest in Figure 2.2 and we'd like to map it to a 2-parking function, then we start with indexing each vertex. The rule is that we start from the first level, i.e. the root and start from the left, then we index the vertices 1, 2, \dots as follows with indexes in the bracket:
\begin{figure}[h!]
  \center
  \includegraphics[width=0.75\textwidth]{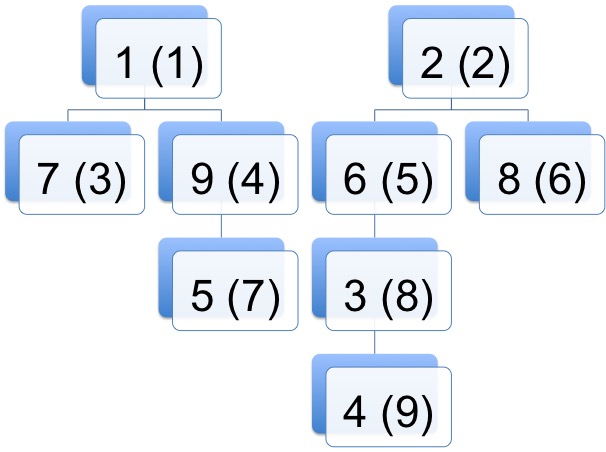}
  \caption{The Indexed Labelled Rooted Forest}
\end{figure}

Now let the first line still be 3456789. For each of them in the first line, the second line number should be the index of its parent. Thus we have
\begin{align*}
vertices: && 3 && 4 && 5 && 6 && 7 && 8 && 9 \\
2-parking function: && 5 && 8 && 4 && 2 && 1 && 2 && 1
\end{align*}
which is exactly (\#).

There are other bijections. For example, in the paper \cite{DV}, the authors define a bijection without using recurrence between the set of $a$-parking functions of length $n$ to the set of rooted labelled forests with $a$ components and $a+n$ vertices, for which 
$$
{{n+1} \choose 2} - Sum(p_1, \dots, p_n) = inv(F),
$$
where $Sum(p_1, \dots, p_n)$ is the sum statistic of parking functions defined in Section 2.5, $inv(F)$ is the inversion of a forest $F$, and the parking function $(p_1, \dots, p_n)$ is mapped to the forest $F$ (vice versa). 
%%%%%%%%%%%%%%%%%%%%%

\section{From Enumeration to Statistics}

Often in enumerative combinatorics, the class of interest has natural `statistics', like height, weight, and IQ for humans, and
one is interested rather than, for a finite set $A$,
$$
|A| \, := \, \sum_{a \in A} 1 ,
$$
called the {\it naive counting}, and getting a {\bf number} (obviously a non-negative integer), by the so-called {\it weighted counting},
$$
|A|_x \, := \,\sum_{a \in A} x^{f(a)} ,
$$
where $f:=A \rightarrow Z$ is the statistic in question. To go  from the weighted enumeration (a certain Laurent polynomial)
to straight enumeration, one sets $x=1$, i.e. $|A|_1 = |A|$.

Since this is {\it mathematics}, and not {\it accounting}, the usual scenario is not just {\bf one} specific set
$A$, but a sequence of sets $\{A_n\}_{n=0}^{\infty}$, and then the enumeration problem is to have an efficient
description of the numerical sequence $a_n:=|A_n|$, ready to be looked-up (or submitted) to the OEIS, and
its corresponding sequence of polynomials $P_n(x):=|A_n|_x$.

It often happens that the statistic $f$, defined on $A_n$, has a {\it scaled limiting distribution}.
In other words, if you draw a {\it histogram} of $f$ on $A_n$,, and do the
obvious {\it scaling}, they get closer and closer to a certain {\it continuous} curve,  as $n$ goes to infinity.

The scaling is as follows. Let $E_n(f)$ and $Var_n(f)$ the {\it expectation} and
{\it variance} of the statistic $f$ defined on $A_n$, and define the {\it scaled} random variable, for $a\in A_n$, by
$$
X_n (a):= \frac{f(a)- E_n(f)}{\sqrt{Var_n(f)}}  .
$$

If you draw the histograms of $X_n(a)$ for large $n$, they look practically the same, and converge to some {\it continuous limit}.

A famous example is {\it coin tossing}. If $A_n$ is $\{-1,1\}^n$, and $f(v)$ is the sum of $v$, then the limiting distribution
is the {\it bell shaped  curve} aka {\it standard normal distribution} aka {\it Gaussian distribution}.

As explained in \cite{54}, a purely finitistic approach to finding, and proving, a limiting scaled distribution, is via the
{\it method of moments}. Using {\it symbolic computation}, the computer can rigorously prove {\it exact} expressions
for as many moments as desired, and often (like in the above case, see \cite{54}) find a recurrence for the
sequence of moments.  This enables one to identify the limits of the scaled moments with the moments of the continuous limit
(in the example of coin-tossing [and many other cases], $\frac{e^{-x^2/2}}{\sqrt{2\pi}}$, 
whose moments are famously $1 , 0 , 1\cdot 3, 0, 1 \cdot 3 \cdot 5, 0 , 1 \cdot 3 \cdot 5 \cdot 7, 0, \dots$) .
Whenever this is the case the discrete family of random variables is called {\it asymptotically normal}.
Whenever this is {\bf not} the case, it is interesting and surprising.

\section{The Sum and Area Statistics of Parking Functions}

Let $\mathcal{P}(n,a)$ be the set of $a$-parking functions of length $n$.

A natural statistic is the sum
$$
Sum(p_1, \dots, p_n) := p_1 + p_2 + \dots + p_n =\sum_{i=1}^{n} p_i .
$$

Another, even more natural (see the beautiful article \cite{7}), happens to be
$$
Area(p):= \frac{n(2a+n-1)}{2} -Sum(p)  .
$$

Let $P(n,a)(x)$ be the weighted analog of $p(n,a)$, according to {\it Sum}, i.e.
$$
P(n,a)(x) \, := \, \sum_{p \in  \mathcal{P}(n,a)} x^{Sum(p)}  .
$$

Analogously, let $Q(n,a)(x)$ be the weighted analog of $p(n,a)$, according to {\it Area}, i.e.
$$
Q(n,a)(x) \, := \, \sum_{p \in  \mathcal{P}(n,a)} x^{Area(p)}.
$$

Clearly, one can easily go from one to the other
$$
Q(n,a)(x) \, = \, x^{(2a+n-1)n/2} \, P(n,a)(x^{-1})  , \quad
P(n,a)(x) \, = \, x^{(2a+n-1)n/2} \, Q(n,a)(x^{-1}) .
$$

How do we compute $P(n,a)(x)$?, (or equivalently, $Q(n,a)(x)$?). It is readily seen that the analog of
Fundamental Recurrence for the weighted counting is
$$
P(n,a)(x) \, = \, x^n \, \sum_{k=0}^{n} \, {{n} \choose {k}} P(n-k,a+k-1)(x) ,
$$
subject to the initial conditions $P(0,a)(x)=1$ and $P(n,0)(x)=0$.

So it is {\it almost} the same, the ``only" change is sticking $x^n$ in front of the sum on the right hand side.

Equivalently, 
$$
Q(n,a)(x) \, = \, \, \sum_{k=0}^{n} \, {{n} \choose {k}} x^{k(k+2a-3)/2} \,  Q(n-k,a+k-1)(x)  ,
$$
subject to the initial conditions $Q(0,a)(x)=1$ and $Q(n,0)(x)=0$.

Once again, in the sense of Wilf, this is already an {\it answer}, but because of the extra variable $x$, one
can not go as far as we did before for the naive, merely numeric, counting.

It is very unlikely that there is a ``closed form'' expression for $P(n,a)(x)$ (and hence $Q(n,a)(x)$), but for
{\it statistical purposes} it would be nice to get ``closed form'' expressions for

$\bullet$ the expectation,

$\bullet$ the variance,

$\bullet$  as many factorial moments as possible, from which the `raw' moments, and latter the {\it centralized} moments
and finally the {\it scaled moments} can be gotten. Then we can take the limits as $n$ goes to infinity, and
see if they match the moments of any of the known continuous distributions, and prove {\it rigorously}
that, at least for that many moments, the conjectured limiting distribution  matches.

In our case, the limiting distribution is the intriguing so-called {\it Airy distribution}, that Svante Janson prefers
to call ``the area under Brownian excursion". This result was stated and proved in \cite{7}, by using deep and sophisticated
continuous probability theory and continuous martingales. Here we will ``almost" prove this result, in the sense of
showing that the limits of the scaled moments of the area statistic on parking functions coincide
with the scaled moments of the Airy distribution up to the $30$-th moment, and we can go much further.

But we can do much more than continuous probabilists. We (or rather our computers, running Maple) can
find {\it exact} {\bf polynomial} expressions in $n$ and the expectation $E_1(n)$. 
We can do it for any desired number of moments, say $30$. Unlike continuous probability theorists,
our methods are entirely elementary, only using {\it high school algebra}.

We can also do the same thing for the more general $a$-parking functions. Now the expressions are
polynomials in $n$, $a$, and the expectation $E_1(n,a)$.

Finally, we believe that our approach, using the  fundamental recurrence of area statistic, can be used
to give a full proof (for all moments), by doing it asymptotically, and deriving a recurrence
for the leading terms of the asymptotics for the factorial moments that would coincide with the
well-known recurrence for the moments of the Airy distribution given, for example in  Eqs. (4) and (5) of
Svante Janson's article \cite{22}.  This is left as a challenge to our readers.

The expectation of the sum statistic, let's call it $E_{sum}(n,a)$ is given by 
$$
E_{sum}(n,a) \, = \, \frac{P'(n,a)(1)}{P(n,a)(1)} \,  = \,  \frac{P'(n,a)(1)}{a(a+n)^{n-1}} ,
$$
where the prime denotes, as usual, differentiation w.r.t. $x$.

Can we get a {\it closed-form expression} for $P'(n,a)(1)$, and hence for $E_{sum}(n,a)$?

Differentiating the fundamental recurrence of $P(n,a)(x)$ with respect to $x$, using the product rule, we get
$$
P(n,a)'(x) \, =\, x^n \, \sum_{k=0}^{n} \, {{n} \choose {k}}  P(n-k,a+k-1)'(x) \, + \,
 \, n x^{n-1} \, \sum_{k=0}^{n} \, {{n} \choose {k}}  \, P(n-k,a+k-1)(x)  .
$$
Plugging-in $x=1$ we get that $P(n,a)'(1)$  satisfies the recurrence 
$$
P(n,a)'(1) - \sum_{k=0}^{n} \, {{n} \choose {k}} P(n-k,a+k-1)'(1) =
n\, \sum_{k=0}^{n} \, {{n} \choose {k}} P(n-k,a+k-1)(1) = n\,p(n,a)  .
$$

Using this recurrence, we can, just as we did for $p(n,a)$ above, get  expressions, as polynomials in $a$,
for numeric $1 \leq n \leq 10$, say, and then conjecture that
$$
P'(n,a)(1)= \frac{1}{2} \, a\,n \, (a+n-1) \, (a+n)^{n-1}-
\frac{1}{2} \sum_{j=1}^{n} {{n} \choose {j}} \, j! \, a \, (a+n)^{n-j} .
$$
To prove it, one plugs in the left side into the above recurrence of $P(n,a)'(1)$, changes the order of summation, and
simplifies. This is rather tedious, but since at the {\it end of the day}, these are equivalent to
{\it polynomial} identities in $n$ and $a$, checking it for sufficiently many special values of $n$ and $a$
would be a rigorous proof.

It follows that
$$
E_{sum}(n,a) \,= \,
\frac{n(a+n+1)}{2} - \frac{1}{2} \, \sum_{j=1}^{n} \frac{n!}{(n-j)! (a+n)^{j-1}}  .
$$
This formula first appears in \cite{30}.

Equivalently,

$$
E_{area}(n,a) \,= \,
\frac{n\,(a-2)}{2} \,+ \,\frac{1}{2} \, \sum_{j=1}^{n} \frac{n!}{(n-j)! (a+n)^{j-1}}  .
$$
In particular, for the primary object of interest, the case $a=1$, we get
$$
E_{area}(n,1) \,= \,
- \frac{n}{2} \,+ \,\frac{1}{2} \, \sum_{j=1}^{n} \frac{n!}{(n-j)! (n+1)^{j-1}} .
$$
This rings a bell! It may written as
$$
E_{area}(n,1) \,= \,- \frac{n}{2} \,+ \, \frac{1}{2} W_{n+1} ,
$$
where $W_n$ is the {\bf iconic} quantity, 
$$
W_n \, = \, \frac{n!}{n^{n-1}} \sum_{k=0}^{n-2} \frac{n^k}{k!} ,
$$
proved by Riordan and Sloane \cite{34} to be the expectation of another very important quantity, the sum of the heights
on rooted labeled trees on $n$ vertices. In addition to its considerable mathematical interest, this quantity, $W_n$,
has great {\it historical significance}, it was the {\it first sequence}, sequence $A435$
of the amazing On-Line Encyclopedia of Integer Sequences
(OEIS), now with more than  $300000$ sequences! See \cite{11} for details, and far-reaching extensions, analogous to the present chapter.

[{The reason it is not sequence A1  is that initially the sequences were arranged in lexicographic order.}]

Another  fact, that will be of great use later in this chapter, is that, as noted in \cite{34}, Ramanujan and Watson proved that
 $W_n$ (and hence $W_{n+1}$) is asymptotic to 
$$
\frac{\sqrt{2\pi}}{2} \, n^{3/2}  .
$$

It is very possible that the formula $E_{area}(n,1) \,= \,- \frac{n}{2} \,+ \, \frac{1}{2} W_{n+1}$ may also be
deduced from the Riordan-Sloane result via one of the numerous known bijections between parking functions
and rooted labeled trees. More generally, the results below, for the special case $a=1$, might be deduced, from
those of \cite{11}, but we believe that the present {\it methodology} is interesting for its own sake, and besides
in our {\it current} approach (that uses recurrences rather than the Lagrange Inversion Formula),
it is much faster to compute higher moments, hence, going in the other direction, would produce many more moments for
the statistic on rooted labeled trees considered in \cite{11}, provided that there is indeed such a correspondence
that sends the area statistic on parking functions (suitably tweaked) to the Riordan-Sloane statistic on rooted labeled trees.

\section{The Limiting Distribution}

Given a combinatorial family, one can easily get an idea of the limiting distribution
by taking a large enough $n$, say $n=100$, and generating a large enough number of random
objects, say $50000$, and drawing a {\it histogram}, see Figure 2 in Diaconis and
Hicks' insightful article \cite{7}. But, one does not have to resort to simulation.
While it is impractical to consider {\it all} $101^{99}$ parking functions of length $100$, the
generating function $Q(100,1)(x)$ contains the exact count for each
conceivable area from $0$ to ${{100} \choose {2}}$. 

\begin{figure}[h!]
  \center
  \includegraphics[width=0.75\textwidth]{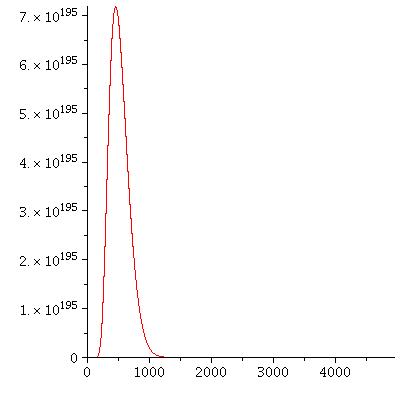}
  \caption{A Histogram of the Area of Parking Functions of Length 100}
\end{figure}

But an even more informative way to investigate the limiting distribution
 is to draw the histogram of the probability generating function
of the scaled distribution 
$$
X_n (p):= \frac{Area(p)- E_n}{\sqrt{Var_n}}  ,
$$
where $E_n$ and $Var_n$ are the expectation and variance respectively.

\begin{figure}[h!]
  \center
  \includegraphics[width=0.75\textwidth]{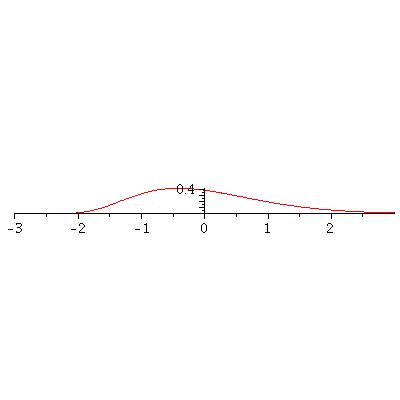}
  \caption{The Scaled Distribution of the Area of Parking Functions of Length 100}
\end{figure}

\begin{figure}[h!]
  \center
  \includegraphics[width=0.75\textwidth]{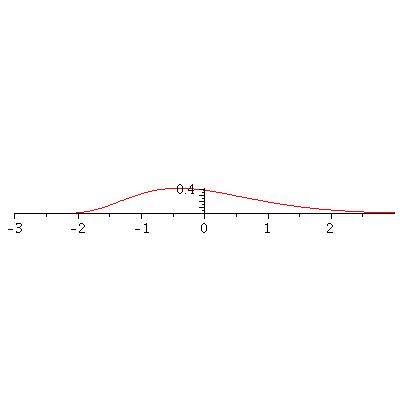}
  \caption{The Scaled Distribution of the Area of Parking Functions of Length 120}
\end{figure}

As proved in \cite{7} (using deep results in continuous probability due to David Aldous, Svante Janson,
and Chassaing and Marcket) the limiting distribution is the {\it Airy distribution}.
We will soon ``almost" prove it, but do much more by discovering exact expressions for the first $30$ moments,
not just their limiting asymptotics.

\section{Truly Exact Expressions for the Factorial Moments}

In \cite{31} there is an ``exact'' expression for the general moment, that is not very useful for our purposes.
If one traces their proof, one can, conceivably, get explicit expressions for each specific moment,
but they did not bother to implement it, and the asymptotics are not immediate.

We discovered,  the following important fact.

{\bf Fact.} Let $E_1(n,a):=E_{area}(n,a)$ be the expectation of the area statistic on $a$-parking functions
of length $n$,  given above, and let $E_k(n,a)$ be the $k$-th factorial moment
$$
E_k(n,a) \, := \, \frac{ Q^{(k)}(n,a)(1)}{a(a+n)^{n-1}} ,
$$
then there exist polynomials $A_k(n,a)$ and $B_k(n,a)$ such that
$$
E_k(n,a) \, = A_k(n,a) \,+\,  B_k(n,a) \, E_1(n,a)  .
$$

The beauty of experimental mathematics is that these can be found by cranking out enough data, using
the sequence of probability generating functions $Q(n,a)(x)$, obtained by using the fundamental recurrence of area statistic, getting sufficiently many numerical data for the moments,
and using {\it undetermined coefficients}. These can be proved {\it a posteriori} by taking
these truly exact formulas and  verifying that the implied recurrences for the $k$-th factorial moment,
in terms of the previous ones. But this is {\bf not} necessary. Since, at the end of the day,
it all boils down to verifying {\bf polynomial identities}, so, once again, verifying them
for sufficiently many different values of $(n,a)$ constitutes a rigorous proof. To be fully rigorous,
one  needs to prove {\it a priori} bounds for the degrees in $n$ and $a$, but, in our humble opinion,
it is not that important, and could be left to the obtuse reader.

\begin{theorem}[equivalent to a result in \cite{30}]
The expectation of the area statistic on parking functions of length $n$ is
$$
E_1(n):= \,- \frac{n}{2} \,+ \, \frac{1}{2} \, \frac{(n+1)!}{(n+1)^{n}} \sum_{k=0}^{n-1} \frac{(n+1)^k}{k!}  ,
$$
and asymptotically it equals $\frac{\sqrt{2\pi}}{4} \cdot n^{3/2} +O(n)$.
\end{theorem}

\begin{theorem}
The {\bf second} factorial moment of the area statistic on parking functions of length $n$ is
$$
 -\frac{7}{3} (n+1) \, E_1(n)+{\frac {5}{12}}\,{n}^{3}- \frac{1}{12} \,{n}^{2}- \frac{1}{3} \,n  ,
$$
and asymptotically it equals $\frac{5}{12} \cdot n^3 +O(n^{5/2})$.
\end{theorem}

\begin{theorem}
The {\bf third} factorial moment of the area statistic on 
parking functions of length $n$ is
$$
-{\frac {175}{192}}\,{n}^{4}-{\frac {283}{192}}\,{n}^{3}+{\frac {199}{
192}}\,{n}^{2}+{\frac {259}{192}}\,n+ \left( {\frac {15}{32}}\,{n}^{3}+
{\frac {521}{96}}\,{n}^{2}+{\frac {1219}{96}}\,n+{\frac {743}{96}} 
 \right) \, E_1(n)  ,
$$
and asymptotically it equals $\frac{15}{128}\sqrt{2\pi} \cdot n^{9/2} +O(n^{4})$.
\end{theorem}

\begin{theorem}
 The {\bf fourth} factorial moment of the area statistic on 
parking functions of length $n$ is
$$
{\frac {221}{1008}}\,{n}^{6}+{\frac {63737}{30240}}\,{n}^{5}+{\frac {
101897}{15120}}\,{n}^{4}+{\frac {22217}{5040}}\,{n}^{3}-{\frac {1375}{
189}}\,{n}^{2}-{\frac {187463}{30240}}\,n
$$
$$
+ \left( -{\frac {35}{16}}\,{n
}^{4}-{\frac {449}{27}}\,{n}^{3}-{\frac {130243}{2520}}\,{n}^{2}-{
\frac {7409}{105}}\,n-{\frac {503803}{15120}} \right) \, E_1(n) ,
$$
and asymptotically it equals $\frac{221}{1008}\cdot n^{6} +O(n^{11/2})$.
\end{theorem}

\begin{theorem}
The {\bf fifth} factorial moment of the area statistic on 
parking functions of length $n$ is
$$
-{\frac {105845}{110592}}\,{n}^{7}-{\frac {2170159}{290304}}\,{n}^{6}-{
\frac {99955651}{3870720}}\,{n}^{5}-{\frac {30773609}{725760}}\,{n}^{4}
$$
$$
-{\frac {94846903}{11612160}}\,{n}^{3}+{\frac {24676991}{483840}}\,{n}^
{2}
+{\frac {392763901}{11612160}}\,n
$$
$$
+ ( {\frac {565}{2048}}\,{n}^{
6}+{\frac {1005}{128}}\,{n}^{5}+{\frac {9832585}{165888}}\,{n}^{4}+{
\frac {1111349}{5184}}\,{n}^{3}+{\frac {826358527}{1935360}}\,{n}^{2}
$$
$$
+{
\frac {159943787}{362880}}\,n+{\frac {1024580441}{5806080}} )  E_1(n) ,
$$
and asymptotically it equals $\frac{565}{8192    }\sqrt{2\pi} \cdot n^{15/2} +O(n^{7})$.
\end{theorem}

\begin{theorem}
 The {\bf sixth} factorial moment of the area statistic of
parking functions of length $n$ is
$$
{\frac {82825}{576576}}\,{n}^{9}+{\frac {373340075}{110702592}}\,{n}^{8
}+{\frac {9401544029}{332107776}}\,{n}^{7}+{\frac {14473244813}{
127733760}}\,{n}^{6}+{\frac {414139396709}{1660538880}}\,{n}^{5}
$$
$$
+{
\frac {88215445651}{332107776}}\,{n}^{4}-{\frac {18783816473}{332107776
}}\,{n}^{3}-{\frac {643359542029}{1660538880}}\,{n}^{2}-{\frac {
358936540409}{1660538880}}\,n
$$
$$
+ ( -{\frac {3955}{2048}}\,{n}^{7}-{
\frac {186349}{6144}}\,{n}^{6}-{\frac {259283273}{1161216}}\,{n}^{5}-{
\frac {119912501}{129024}}\,{n}^{4}-{\frac {149860633081}{63866880}}\,{
n}^{3}
$$
$$
-{\frac {601794266581}{166053888}}\,{n}^{2}-{\frac {864000570107}
{276756480}}\,n-{\frac {921390308389}{830269440}} ) \, E_1(n)  ,
$$
and asymptotically it equals $\frac{ 82825}{576576}\cdot n^{9} +O(n^{17/2})$.
\end{theorem}

For Theorems 7-30, see the output file

\noindent {\tt http://sites.math.rutgers.edu/\~{}zeilberg/tokhniot/oParkingStatistics7.txt} 

Let $\{e_k\}_{k=1}^{\infty}$ be the sequence of moments of the Airy distribution, defined by the recurrence given
in Equations $(4)$ and $(5)$ in Svante Janson's interesting survey paper \cite{22}. Our computers, using our
Maple package, proved that 
$$
E_k(n) \, = \, e_k n^{\frac{3k}{2}} + O( n^{\frac{3k-1}{2}}) ,
$$
for $1 \leq k \leq 30$. It follows that the {\it limiting distribution} of the area statistic is (most probably) the Airy distribution,
since the first $30$ moments match.
Of course, this was already known to continuous probability theorists, and we only proved it for the first $30$
moments, but:

$\bullet$ Our methods are purely {\it elementary} and {\it finitistic}.

$\bullet$ We can easily go much farther, i.e. prove it for more moments.

$\bullet$ We believe that our approach, using recurrences, can be used to derive a recurrence for the {\it leading}
asymptotics of the factorial moments, $E_k(n)$, that would turn out to be the same as the above mentioned
recurrence (Eqs. (4) and (5) in \cite{22}). We leave this as a challenge to the reader.

We also have the results of the exact expressions for the first $10$ moments of the area statistic for general $a$-parking. To see expressions in $a$, $n$, and $E_1(n,a)$, for the first $10$ moments of $a$-parking, see

\noindent {\tt http://sites.math.rutgers.edu/\~{}zeilberg/tokhniot/oParkingStatistics8.txt}.

\chapter{The Gordian Knot of the $C$-finite Ansatz}

This chapter is adapted from \cite{48}, which has been accepted on {\it Algorithmic Combinatorics-Enumerative Combinatorics, Special Functions, and Computer Algebra: In honor of Peter Paule's 60th birthday}. It is also available on arXiv.org, number 1812.07193.

This chapter is dedicated to Peter Paule, one of the great pioneers of experimental mathematics
and symbolic computation. In particular, it is greatly inspired by his masterpiece,
co-authored with Manuel Kauers, {\it The Concrete Tetrahedron} \cite{24}, where a whole chapter
is dedicated to our favorite {\it ansatz}, the $C-$finite ansatz.

\section{Introduction}

Once upon a time there was a knot that no one could untangle, it was so complicated. Then
came Alexander the Great and, in one second, {\it cut} it with his sword.

Analogously, many mathematical problems are very hard, and the current party line is
that in order for it be considered solved, the solution, or answer, should be given
a logical, rigorous, deductive proof. 

Suppose that you want to answer the following question: 

\vspace{1mm}\noindent

{\it Find a closed-form formula, as an expression in $n$,
for the real part of the $n$-th complex root of the Riemann zeta function, $\zeta(s)$ .} 

\vspace{1mm}\noindent
Let's call this quantity $a(n)$. Then you compute these real numbers, and find out
that $a(n)=\frac{1}{2}$ for $n \leq 1000$. Later you are told by Andrew Odlyzko that $a(n)=\frac{1}{2}$ for all $1 \leq n \leq 10^{10}$.
Can you conclude that $a(n)=\frac{1}{2}$ for {\it all} $n$? we would, but, at this time of writing, there is
no way to deduce it rigorously, so it remains an open problem. 
It is very possible that one day  it will turn out that $a(n)$ (the real part of the $n$-th complex root of $\zeta(s)$)
belongs to a certain {\it ansatz}, and that checking it for the first $N_0$ cases implies its truth
in general, but this remains to be seen.

There are also frameworks, e.g. {\it Pisot sequences} (see \cite{10}, \cite{58}), where the inductive approach fails miserably.

On the other hand, in order to (rigorously) prove that $1^3+2^3+3^3+ \dots + n^3=(n(n+1)/2)^2$, 
for {\it every} positive integer $n$, it suffices to check it
for the five special cases $0 \leq n \leq 4$, since both sides are polynomials of {\bf degree} $4$, hence the
difference is a polynomial of degree $\leq 4$, given by five `degrees of freedom'. 

This is an example of what is called  the  `$N_0$ principle'. In the case of a polynomial identity (like this one),
$N_0$ is simply the degree plus one.

But our favorite {\it ansatz} is the $C$-finite ansatz. A sequence of numbers $\{a(n)\}$ ($0 \leq n < \infty$)
is $C$-finite if it satisfies a {\it linear recurrence equation with constant coefficients}.
For example the Fibonacci sequence that satisfies $F(n)-F(n-1)-F(n-2)=0$ for $n \geq 2$.

The $C$-finite ansatz is beautifully described in Chapter 4 of  the masterpiece
{\it The Concrete Tetrahedron} \cite{24}, by Manuel Kauers and Peter Paule, and discussed at length in \cite{57}.

Here the  `$N_0$ principle' also holds (see \cite{59}), i.e. by looking at the `big picture' one can determine
{\it a priori}, a positive integer, often not that large, such that checking that $a(n)=b(n)$ for $1 \leq n \leq N_0$
implies that $a(n)=b(n)$ for all $n>0$.

A sequence $\{a(n)\}_{n=0}^{\infty}$ is $C$-finite if and only if its (ordinary) {\it generating function}
$f(t):=\sum_{n=0}^{\infty} a(n)\,t^n$ is a {\bf rational function} of $t$, i.e.
$f(t)=P(t)/Q(t)$ for some {\it polynomials} $P(t)$ and $Q(t)$. 
For example, famously, the generating function of the Fibonacci
sequence is $t/(1-t-t^2)$.

Phrased in terms of generating functions, the $C$-finite ansatz is the subject of chapter 4 of yet another
masterpiece, Richard Stanley's `Enumerative Combinatorics' (volume 1) \cite{43}. There it is shown, using the
`transfer matrix method' (that originated in physics), that in many combinatorial situations, where
there are finitely many states, one is guaranteed, {\it a priori}, that the generating function is rational.

Alas, finding this transfer matrix, at each specific case, is not easy!
The human has to first figure out the set of states, and then using human ingenuity, figure out
how they interact.

A better way is to automate it. Let the computer do the research, and using `symbolic dynamical programming',
the computer, automatically, finds the set of states, and constructs, {\it all by itself} (without any human
pre-processing), the set of states and the transfer matrix. But this may not be so efficient for two reasons.
First, at the very end, one has to invert a matrix with {\it symbolic} entries, hence compute 
symbolic determinants, that is time-consuming. Second, setting up the `infrastructure' and writing a program
that would enable the computer to do `machine-learning' can be very daunting.

In this chapter, we will describe two {\it case studies} where, by `general nonsense', we know
that the generating functions are rational, and it is easy to bound the degree of the denominator
(alias the order of the recurrence satisfied by the sequence).  
Hence a simple-minded, {\it empirical}, approach of computing the first few terms and then
`fitting' a recurrence (equivalently rational function) is possible.

The first case-study concerns
counting spanning trees in families of grid-graphs, studied by Paul Raff \cite{33}, and F.J. Faase \cite{14}.
In their research, the human first analyzes the intricate combinatorics, manually sets up
the transfer matrix, and only at the end lets a computer-algebra system evaluate the symbolic determinant.

Our key observation, that enabled us to `cut the Gordian knot' is that  the terms of the studied sequences are
expressible as {\it numerical} determinants. Since computing numerical determinants is so fast,
it is easy to compute sufficiently many terms, and then fit the data into a rational function.
Since we easily have an upper bound for the degree of the denominator of the rational function,
everything is rigorous.

The second case-study is computing generating functions for sequences of determinants of `almost diagonal matrices'.
Here, in addition to the `naive' approach of cranking enough data and then fitting it into a rational function, we
also describe the `symbolic dynamical programming method', that surprisingly, is faster for the range
of examples that we considered. But we believe that for sufficiently large cases, the naive approach will
eventually be more efficient, since the `deductive' approach works equally well for the analogous problem of
finding the sequence of permanents of these almost diagonal matrices, for which the naive approach will soon be
intractable.

This chapter may be viewed as a {\it tutorial}, hence we include lots of implementation details, and
Maple code. We hope that it will  inspire readers (and their computers!) to apply it in other situations

%%%%%%%%%%%%%%%%%%%%%%%%%%%%%%%%%%%%%%%%%%%%%%%%%%%%%%%%%%%%%%%%%%%%%%%%%%%%%%%%%%%%%%%%%%%%%%%

\section{The Human Approach to Enumerating Spanning Trees of Grid Graphs}

In order to illustrate the advantage of ``keeping it simple'', we will review  the human approach to 
the enumeration task that we will later redo using the `Gordian knot' way.
While the human approach is definitely interesting for its own sake, it is rather painful.

Our goal is to enumerate the number of spanning trees 
in certain families of graphs, 
notably grid graphs and their generalizations. Let's examine  Paul Raff's interesting approach
described in his paper {\it Spanning Trees in Grid Graph} \cite{33}. Raff's approach was inspired by 
the pioneering  work of F. J. Faase \cite{14}.

The goal is to find generating functions that
enumerate spanning trees in grid graphs and the product of an arbitrary graph and a path or a cycle. 

Grid graphs have two parameters, let's call them, $k$ and $n$. 
For a $k \times n$ grid graph, let's think of $k$ as {\it fixed} while $n$ is the discrete input variable of interest.

\begin{definition}
 The $k \times n$ grid graph $G_k(n)$ is the following graph given in terms of its vertex
set $V$ and edge set $E$:
$$
V = \{v_{ij}|1 \leq i \leq k, 1 \leq j \leq n\},
$$
$$
S = \{\{v_{ij}, v_{i'j'}\}| |i-i'|+|j-j'|=1\}.
$$
\end{definition}

The main idea in the human approach is to consider  the collection of set-partitions  of
$[k] = \{1,2,\dots,k\}$ and figure out  the transition when we extend a $k \times n$ grid graph to a $k \times (n+1)$ one.

Let $\mathcal{B}_k$ be the collection  of all set-partitions of $[k]$.  
$B_k = |\mathcal{B}_k|$ are called the Bell number. Famously, the exponential generating function
of $B_k$, namely $\sum_{k=0}^{\infty} \frac{B_k}{k!}\, t^k$, equals $e^{e^t-1}$.

A lexicographic ordering on $\mathcal{B}_k$ is defined as follows:

\begin{definition}
Given two partitions $P_1$ and $P_2$ of $[k]$, for $i \in [k]$, let $X_i$ be the block of $P_1$ containing $i$ and $Y_i$ be the 
block of $P_2$ containing $i$. Let $j$ be the minimum number such that $X_i \neq Y_i$. Then $P_1 < P_2$ iff

1. $|P_1| < |P_2|$ or

2. $|P_1| = |P_2|$ and $X_j \prec Y_j$ where $\prec$ denotes the normal lexicographic order. 
\end{definition}

For example, here is the ordering for $k=3$:
$$
\mathcal{B}_3 = \{\{\{1,2,3\}\}, \{\{1\}, \{2,3\}, \{\{1,2\}, \{3\}\}, \{\{1,3\}, \{2\}\}, \{\{1\}, \{2\}, \{3\}\}\} .
$$
For simplicity, we can rewrite it as follows:
$$
\mathcal{B}_3 = \{123, 1/23, 12/3, 13/2, 1/2/3\}.
$$

\begin{definition}
Given a spanning forest $F$ of $G_k(n)$, the partition induced by $F$ is obtained from the equivalence relation

\centerline{$i \sim j \Longleftrightarrow v_{n,i}, v_{n,j} $ are in the same component of $F$.}
\end{definition}

For example, the partition induced by any spanning tree of $G_k(n)$ is $123\dots k$ because by definition,
in a spanning tree, all $v_{n,i}, 1 \leq i \leq k$ are in the same component.
For the other extreme,
where every component only consists of one vertex, the corresponding set-partition is
$1/2/3/\dots /k-1/k$ because no two $v_{n,i}, v_{n,j}$ are in the same component for $1 \leq i<j \leq k$.

\begin{definition}
Given a spanning forest $F$ of $G_k(n)$ and a set-partition $P$ of $[k]$, we say that $F$ is consistent with $P$ if:

1. The number of trees in $F$ is precisely $|P|$.

2. $P$ is the partition induced by $F$.
\end{definition}

Let $E_n$ be the set of edges $E(G_k(n) \backslash E(G_k(n-1))$, then $E_n$ has $2k-1$ members.

Given a forest $F$ of $G_k(n-1)$ and some subset $X \subseteq E_n$, we can combine them to 
get a forest of $G_k(n)$ as follows. 
We just need to know  how many subsets of $E_n$ can transfer a forest consistent with some partition to a forest consistent with another partition. 
This leads to the following definition:

\begin{definition}
Given two partitions $P_1$ and $P_2$ in $\mathcal{B}_k$, a subset $X \subseteq E_n$ transfers from $P_1$ to $P_2$ if a forest 
consistent with $P_1$ becomes a forest consistent with $P_2$ after the addition of $X$. In this case, we write $X \diamond P_1 = P_2$.
\end{definition}

With the above definitions, it is natural to define a $B_k \times B_k$ transfer matrix $A_k$ by the following:
$$
A_k(i,j) = | \{A \subseteq E_{n+1} | A \diamond P_j = P_i \} |.
$$
Let's look at the $k=2$ case as an example. We have 
$$
\mathcal{B}_2 = \{12, 1/2\},\quad E_{n+1} = \{\{v_{1,n}, v_{1,n+1}\}, \{v_{2,n}, v_{2,n+1}\}, \{v_{1,n+1}, v_{2,n+1}\}\}.
$$
For simplicity, let's call the edges in $E_{n+1}$ $e_1, e_2, e_3$. Then to transfer the set-partition $P_1 = 12$ to itself, 
we have the following three ways: $\{e_1, e_2\}, \{e_1, e_3\}, \{e_2, e_3\}$. 
In order to transfer the partition $P_2=1/2$ into $P_1$, we  only have one way, namely:
$\{e_1, e_2, e_3\}$. Similarly, there are two ways to transfer $P_1$ to $P_2$ and one way to transfer $P_2$ to itself 
Hence the transfer matrix is the following $2 \times 2$ matrix:
$$
A =
\begin{bmatrix}
 3 & 1 \\
 2 & 1
\end{bmatrix}.
$$
Let $T_1(n), T_2(n)$ be the number of forests of $G_k(n)$ which are consistent with the partitions 
$P1$ and  $P2$, respectively. Let
$$
v_n =
\begin{bmatrix}
 T_1(n) \\
 T_2(n)
\end{bmatrix} ,
$$
then
$$
 v_n =Av_{n-1}  .
$$
The characteristic polynomial of $A$ is
$$
\chi_\lambda(A) = \lambda^2-4\lambda+1.
$$
By the Cayley-Hamilton Theorem, $A$ satisfies 
$$
A^2-4A+1=0.
$$
Hence the recurrence relation for $T_1(n)$ is
$$
T_1(n) = 4T_1(n-1) - T_1(n-2),
$$
the sequence is $\{1, 4, 15, 56, 209, 780, 2911, 10864, 40545, 151316, \dots \}$ (OEIS A001353) and the generating function is 
$$
\frac{x}{1-4x+x^2}.
$$
Similarly, for the $k=3$ case, the transfer matrix
$$
A_3 =
\begin{bmatrix}
 8 & 3 & 3 & 4 & 1 \\
 4 & 3 & 2 & 2 & 1 \\
 4 & 2 & 3 & 2 & 1 \\
 1 & 0 & 0 & 1 & 0 \\
 3 & 2 & 2 & 2 & 1 
\end{bmatrix}.
$$
The transfer matrix method can be generalized to general graphs of the form $G \times P_n$, especially cylinder graphs. 

As one can see, 
we had to think very hard.
First we had to establish a `canonical'
ordering over set-partitions, then
define the consistence between partitions and forests, then look for the transfer matrix and finally worry about initial conditions. 

Rather than think so hard, let's compute sufficiently many terms of the enumeration sequence, and try to guess a linear
recurrence equation with constant coefficients, that would be provable {\it a posteriori} just because we know that {\it there exists}  a transfer matrix
without worrying about finding it explicitly. But how do we generate sufficiently many terms?
Luckily, we can use the celebrated {\bf Matrix Tree Theorem}. 

\begin{theorem}[Matrix Tree Theorem] If $A = (a_{ij})$ is the adjacency matrix of an arbitrary graph $G$, then the number of spanning trees is equal to the 
determinant of any co-factor of the Laplacian matrix $L$ of $G$, where 

$$ L = 
\begin{bmatrix}
    a_{12}+\dots+a_{1n} & -a_{12} &  \dots  & -a_{1,n} \\
    -a_{21} & a_{21}+\dots+a_{2n}  & \dots & -a_{2,n} \\
    \vdots & \vdots  & \ddots & \vdots \\
    -a_{n1} & -a_{n2} & \dots  & a_{n1}+\dots+a_{n,n-1}
\end{bmatrix}.
$$
\end{theorem}

For instance, taking the $(n,n)$ co-factor, we have that the number of spanning trees of $G$ equals
$$
\begin{vmatrix}
    a_{12}+\dots+a_{1n} & -a_{12} &  \dots  & -a_{1,n-1} \\
    -a_{21} & a_{21}+\dots+a_{2n}  & \dots & -a_{2,n-1} \\
    \vdots & \vdots  & \ddots & \vdots \\
    -a_{n-1,1} & -a_{n-1,2} & \dots  & a_{n-1,1}+\dots+a_{n-1,n}
\end{vmatrix}.
$$
Since computing determinants for numeric matrices is very fast, we can find the generating functions for the number of spanning trees in grid graphs and more generalized graphs by experimental methods, using the $C$-finite ansatz.

\section{The GuessRec Maple procedure}

Our engine is the Maple procedure  {\tt GuessRec(L)} that resides 
in the Maple packages accompanying this chapter.
Naturally, we need to collect enough data. The input is the data (given as a list) 
and the output is a conjectured recurrence relation derived from that data.

Procedure {\tt GuessRec(L)}  inputs a list, {\tt L}, and attempts to output a linear recurrence equation with constant coefficients
satisfied by the list. It is based on procedure  {\tt GuessRec1(L,d)} that looks for such a recurrence of order $d$.

The output of   {\tt GuessRec1(L,d)} consists of the
the list of initial $d$ values (`initial conditions')
and the recurrence equation represented as a list. 
For instance, if the input is $L = [1,1,1,1,1,1]$ and $d=1$, then the output will be $[[1],[1]]$; 
if the input is $L=[1, 4, 15, 56, 209, 780, 2911, 10864, 40545, 151316]$ as the $k=2$ case for grid graphs and $d=2$, then the output will be 
$[[1, 4], [4, -1]]$. This means that our sequence satisfies the recurrence $a(n)=4a(n-1)-a(n-2)$, subject to the initial conditions
$a(0)=1,a(1)=4$.

Here is the Maple code:

\vspace{1mm}\noindent

{\obeylines
{\tt
GuessRec1:=proc(L,d) local eq,var,a,i,n:
if nops(L)<=2*d+2 then
 print(`The list must be of size >=`, 2*d+3 ):
 RETURN(FAIL):
fi:
var:=$\{$seq(a[i],i=1..d)$\}$:
eq:=$\{$seq(L[n]-add(a[i]*L[n-i],i=1..d),n=d+1..nops(L))$\}$:
var:=solve(eq,var):
if var=NULL then
 RETURN(FAIL):
else
 RETURN([[op(1..d,L)],[seq(subs(var,a[i]),i=1..d)]]):
fi:
end:
}
}

\vspace{1mm}\noindent

The  idea is that having a long enough list $L$ $(|L|>2d+2)$  of data, we use the data after the $d$-th one to 
discover whether there exists a linear recurrence relation, the first $d$ data points being the initial condition. 
With the unknowns $a_1, a_2, \dots, a_d $, we have a linear systems of no less than $d+3$ equations. If there is a solution, 
it is extremely likely that the recurrence relation holds in general. 
The first list of length $d$ in the output constitutes the list of initial conditions while the second list, $R$,
codes the linear recurrence, where $[R[1], \dots R[d]]$ stands for the following recurrence:
$$
L[n] = \sum_{i=1}^d R[i]L[n-i].
$$

Here is the Maple procedure {\tt GuessRec(L)}:
{\obeylines
{\tt
GuessRec:=proc(L) local gu,d:
for d from 1 to trunc(nops(L)/2)-2 do
 gu:=GuessRec1(L,d):
 if gu<>FAIL then
   RETURN(gu):
 fi:
od:
FAIL:
end:
}
}

This procedure inputs a sequence $L$ and tries to guess a recurrence equation with constant coefficients satisfying it. 
It returns the initial values and the recurrence equation as a pair of lists. Since the length of $L$ is limited, the maximum degree of recurrence 
cannot be more than $\lfloor |L|/2-2 \rfloor$. With this procedure, we just need to input 
$L=[1, 4, 15, 56, 209, 780, 2911, 10864, 40545, 151316]$  to get the recurrence (and initial conditions) $[[1, 4], [4, -1]]$.

Once the recurrence relation, let's call it {\tt S},
is discovered, procedure {\tt CtoR(S,t)} 
finds the generating function for the sequence. 
Here is the Maple code:
\vspace{1mm}\noindent
{\obeylines
{\tt
CtoR:=proc(S,t) local D1,i,N1,L1,f,f1,L:
if not (type(S,list) and  nops(S)=2 and type(S[1],list) and 
type(S[2],list) and nops(S[1])=nops(S[2]) and type(t, symbol) ) then
   print(`Bad input`):
   RETURN(FAIL):
fi:
D1:=1-add(S[2][i]*t**i,i=1..nops(S[2])):
N1:=add(S[1][i]*t**(i-1),i=1..nops(S[1])):
L1:=expand(D1*N1):
L1:=add(coeff(L1,t,i)*t**i,i=0..nops(S[1])-1):
f:=L1/D1:
L:=degree(D1,t)+10:
f1:=taylor(f,t=0,L+1):
if expand([seq(coeff(f1,t,i),i=0..L)])<>expand(SeqFromRec(S,L+1)) 
then
print([seq(coeff(f1,t,i),i=0..L)],SeqFromRec(S,L+1)):
 RETURN(FAIL):
else
 RETURN(f):
fi:
end:
}
}
\vspace{1mm}\noindent
Procedure {\tt SeqFromRec} used above (see the package) simply generates many terms using the recurrence.

Procedure {\tt CtoR(S,t)} outputs the rational function in $t$, 
whose coefficients are the members of the $C$-finite sequence $S$. For example: 
$$
{\tt CtoR([[1,1],[1,1]],t)} = \frac{1}{-t^2-t+1}.
$$
Briefly, the idea is that the denominator of the rational function can be easily determined by the recurrence relation and we use the initial 
condition to find the starting terms of the generating function,
then multiply it by the denominator, yielding the numerator.

\section{Application of GuessRec to Enumerating  Spanning Trees of Grid Graphs and  $G \times P_n$}

With the powerful  procedures {\tt GuessRec} and {\tt CtoR}, we are able to find generating functions 
for the number of spanning trees of generalized graphs of the form $G \times P_n$. We will illustrate 
the application of {\tt GuessRec} to finding the generating function for the number of spanning trees in grid graphs. 

First, using procedure {\tt GridMN(k,n)}, we get the  $k \times n$ grid graph.

Then, procedure {\tt SpFn} uses the Matrix Tree Theorem
to evaluate the determinant of the co-factor of the Laplacian matrix of the grid graph which is the number of spanning trees in this particular graph. 
For a fixed $k$, we need to generate a sufficiently long list of data for the number of spanning trees in 
$G_k(n), n \in [l(k), u(k)]$. 
The lower bound $l(k)$ can't be too small since the first several terms are the initial condition; the upper bound $u(k)$ can't be too small as well 
since we need sufficient data to obtain the recurrence relation. 
Notice that there is a symmetry for the recurrence relation, and
to take advantage of this fact,  modified {\tt GuessRec} to get the more efficient {\tt GuessSymRec} 
(requiring less data).
Once the recurrence relation, and the initial conditions, are given, 
applying {\tt CtoR(S,t)} will give the desirable generating function, that, of course, is a rational function of $t$.
All the above is incorporated in 
procedure {\tt GFGridKN(k,t)} which inputs a positive integer $k$ and a symbol $t$, 
and outputs the generating function whose coefficient of $t^n$ is the number of spanning trees in $G_k(n)$, i.e. if we let 
$s(k,n)$ be the number of spanning trees in $G_k(n)$, the generating function 
$$
F_k(t) = \sum_{n=0}^{\infty} s(k,n) t^n.
$$
We now list  the generating functions $F_k(t)$ for $1 \leq k \leq 7$:
Except for $k=7$, these were already found by Raff \cite{33} and Faase \cite{14}, but it is
reassuring that, using our new approach, we got the same output. The case $k=7$ seems to be new.

\begin{theorem}
The generating function for the number of spanning trees in $G_1(n)$ is:
$$
F_1(t) = \frac {t}{1-t}.
$$
\end{theorem}

\begin{theorem}
The generating function for the number of spanning trees in $G_2(n)$ is:
$$
F_2 = \frac {t}{{t}^{2}-4\,t+1}.
$$
\end{theorem}

\begin{theorem}
The generating function for the number of spanning trees in $G_3(n)$ is:
$$
F_3 = \frac {-{t}^{3}+t}{{t}^{4}-15\,{t}^{3}+32\,{t}^{2}-15\,t+1}.
$$
\end{theorem}

\begin{theorem}
The generating function for the number of spanning trees in $G_4(n)$ is:
$$
F_4 = \frac {{t}^{7}-49\,{t}^{5}+112\,{t}^{4}-49\,{t}^{3}+t}{{t}^{8}-56\,{t
}^{7}+672\,{t}^{6}-2632\,{t}^{5}+4094\,{t}^{4}-2632\,{t}^{3}+672\,{t}^
{2}-56\,t+1}.
$$
For $5 \leq k \leq 7$, since the formulas are too long, we present their numerators and denominators separately.
\end{theorem}

\begin{theorem}
The generating function for the number of spanning trees in $G_5(n)$ is:
$$
F_5 = \frac{N_5}{D_5}
$$
where
$$
N_5 = -{t}^{15}+1440\,{t}^{13}-26752\,{t}^{12}+185889\,{t}^{11}-574750\,{t}^
{10}+708928\,{t}^{9}-708928\,{t}^{7} 
$$
$$
+574750\,{t}^{6}-185889\,{t}^{5}+
26752\,{t}^{4}-1440\,{t}^{3}+t,
$$

$$
D_5 = {t}^{16}-209\,{t}^{15}+11936\,{t}^{14}-274208\,{t}^{13}+3112032\,{t}^{
12}-19456019\,{t}^{11}+70651107\,{t}^{10}
$$
$$
-152325888\,{t}^{9}
+196664896\,{t}^{8}-152325888\,{t}^{7}+70651107\,{t}^{6}-19456019\,{t}^{5}
$$
$$
+
3112032\,{t}^{4}-274208\,{t}^{3}+11936\,{t}^{2}-209\,t+1.
$$
\end{theorem}

\begin{theorem}
The generating function for the number of spanning trees in $G_6(n)$ is:
$$
F_6 = \frac{N_6}{D_6}
$$
where
$$
N_6 = {t}^{31}-33359\,{t}^{29}+3642600\,{t}^{28}-173371343\,{t}^{27}+
4540320720\,{t}^{26}-70164186331\,{t}^{25}
$$
$$
+634164906960\,{t}^{24}-
2844883304348\,{t}^{23}-1842793012320\,{t}^{22}+104844096982372\,{t}^{
21}
$$
$$
-678752492380560\,{t}^{20}+2471590551535210\,{t}^{19}-
5926092273213840\,{t}^{18}
$$
$$
+9869538714631398\,{t}^{17}
-11674018886109840\,{t}^{16}+9869538714631398\,{t}^{15}
$$
$$
-
5926092273213840\,{t}^{14}+2471590551535210\,{t}^{13}
-678752492380560
\,{t}^{12}
$$
$$
+104844096982372\,{t}^{11}-1842793012320\,{t}^{10}-
2844883304348\,{t}^{9}
+634164906960\,{t}^{8}
$$
$$
-70164186331\,{t}^{7}+
4540320720\,{t}^{6}-173371343\,{t}^{5}+3642600\,{t}^{4}-33359\,{t}^{3}
+t,
$$

$$
D_6 = {t}^{32}-780\,{t}^{31}+194881\,{t}^{30}-22377420\,{t}^{29}+1419219792
\,{t}^{28}-55284715980\,{t}^{27}
$$
$$
+1410775106597\,{t}^{26}
-24574215822780\,{t}^{25}+300429297446885\,{t}^{24}
$$
$$
-2629946465331120\,{
t}^{23}+16741727755133760\,{t}^{22}
-78475174345180080\,{t}^{21}
$$
$$
+
273689714665707178\,{t}^{20}-716370537293731320\,{t}^{19}
+1417056251105102122\,{t}^{18}
$$
$$
-2129255507292156360\,{t}^{17}+
2437932520099475424\,{t}^{16}
-2129255507292156360\,{t}^{15}
$$
$$
+
1417056251105102122\,{t}^{14}-716370537293731320\,{t}^{13}
+273689714665707178\,{t}^{12}
$$
$$
-78475174345180080\,{t}^{11}+
16741727755133760\,{t}^{10}-2629946465331120\,{t}^{9}
$$
$$
+300429297446885
\,{t}^{8}-24574215822780\,{t}^{7}+1410775106597\,{t}^{6}-55284715980\,
{t}^{5}
$$
$$
+1419219792\,{t}^{4}-22377420\,{t}^{3}+194881\,{t}^{2}-780\,t+1.
$$
\end{theorem}

\begin{theorem}
The generating function for the number of spanning trees in $G_7(n)$ is:
$$
F_7 = \frac{N_7}{D_7}
$$
where
$$
N_7 = -{t}^{47}-142\,{t}^{46}+661245\,{t}^{45}-279917500\,{t}^{44}+
53184503243\,{t}^{43}-5570891154842\,{t}^{42}
$$
$$
+341638600598298\,{t}^{41
}-11886702497030032\,{t}^{40}+164458937576610742\,{t}^{39}
$$
$$
+4371158470492451828\,{t}^{38}-288737344956855301342\,{t}^{37}+
7736513993329973661368\,{t}^{36}
$$
$$
-131582338768322853956994\,{t}^{35}+
1573202877300834187134466\,{t}^{34}
$$
$$
-13805721749199518460916737\,{t}^{
33}+90975567796174070740787232\,{t}^{32}
$$
$$
-455915282590547643587452175\,
{t}^{31}+1747901867578637315747826286\,{t}^{30}
$$
$$
-5126323837327170557921412877\,{t}^{29}+11416779122947828869806142972\,
{t}^{28}
$$
$$
-18924703166237080216745900796\,{t}^{27}+
22194247945745188489023284104\,{t}^{26}
$$
$$
-15563815847174688069871470516
\,{t}^{25}+15563815847174688069871470516\,{t}^{23}
$$
$$
-22194247945745188489023284104\,{t}^{22}+18924703166237080216745900796
\,{t}^{21}
$$
$$
-11416779122947828869806142972\,{t}^{20}+
5126323837327170557921412877\,{t}^{19}
$$
$$
-1747901867578637315747826286\,{
t}^{18}+455915282590547643587452175\,{t}^{17}
$$
$$
-90975567796174070740787232\,{t}^{16}+13805721749199518460916737\,{t}^{
15}
$$
$$
-1573202877300834187134466\,{t}^{14}+131582338768322853956994\,{t}^
{13}
$$
$$
-7736513993329973661368\,{t}^{12}
+288737344956855301342\,{t}^{11}
$$
$$
-
4371158470492451828\,{t}^{10}-164458937576610742\,{t}^{9}
$$
$$+11886702497030032\,{t}^{8}-341638600598298\,{t}^{7}+5570891154842\,{t}
^{6}-53184503243\,{t}^{5}
$$
$$
+279917500\,{t}^{4}-661245\,{t}^{3}+142\,{t}^
{2}+t,
$$

$$
D_7 = {t}^{48}-2769\,{t}^{47}+2630641\,{t}^{46}-1195782497\,{t}^{45}+
305993127089\,{t}^{44}-48551559344145\,{t}^{43}
$$
$$
+5083730101530753\,{t}^
{42}-366971376492201338\,{t}^{41}
$$
$$
+18871718211768417242\,{t}^{40}
-709234610141846974874\,{t}^{39}
$$
$$
+19874722637854592209338\,{t}^{38}-
422023241997789381263002\,{t}^{37}
$$
$$
+6880098547452856483997402\,{t}^{36}
-87057778313447181201990522\,{t}^{35}
$$
$$
+862879164715733847737203343\,{t}
^{34}-6750900711491569851736413311\,{t}^{33}
$$
$$
+41958615314622858303912597215\,{t}^{32}-208258356862493902206466194607
\,{t}^{31}
$$
$$
+828959040281722890327985220255\,{t}^{30}-
2654944041424536277948746010303\,{t}^{29}
$$
$$
+6859440538554030239641036025103\,{t}^{28}-
14324708604336971207868317957868\,{t}^{27}
$$
$$
+24214587194571650834572683444012\,{t}^{26}-
33166490975387358866518005011884\,{t}^{25}
$$
$$
+36830850383375837481096026357868\,{t}^{24}-
33166490975387358866518005011884\,{t}^{23}
$$
$$
+24214587194571650834572683444012\,{t}^{22}-
14324708604336971207868317957868\,{t}^{21}
$$
$$
+6859440538554030239641036025103\,{t}^{20}-
2654944041424536277948746010303\,{t}^{19}
$$
$$
+828959040281722890327985220255\,{t}^{18}-
208258356862493902206466194607\,{t}^{17}
$$
$$
+41958615314622858303912597215
\,{t}^{16}-6750900711491569851736413311\,{t}^{15}
$$
$$
+862879164715733847737203343\,{t}^{14}-87057778313447181201990522\,{t}^
{13}
$$
$$
+6880098547452856483997402\,{t}^{12}-422023241997789381263002\,{t}
^{11}
$$
$$
+19874722637854592209338\,{t}^{10}-709234610141846974874\,{t}^{9}
+18871718211768417242\,{t}^{8}
$$
$$
-366971376492201338\,{t}^{7}+5083730101530753\,{t}^{6}-48551559344145\,{t}^{5}+305993127089\,{t}^{4}
$$
$$
-1195782497\,{t}^{3}+2630641\,{t}^{2}-2769\,t+1.
$$
\end{theorem}

Note that, surprisingly, the degree of the denominator of $F_7(t)$ 
is $48$ rather than the expected $64$ since the first six generating functions' 
denominator have degree $2^{k-1}$, $1 \leq k \leq 6$. 
With a larger computer, one should be able to compute $F_k$ for larger $k$, using this experimental approach. 

Generally, for an arbitrary graph $G$, we consider the number of spanning trees in $G \times P_n$. With the same methodology, 
a list of data can be obtained empirically with which a a generating function follows.

The original motivation for the Matrix Tree Theorem, first discovered by Kirchhoff (of Kirchhoff's laws fame)
came from the desire to efficiently compute joint resistances in an electrical network.

Suppose one is interested in the joint resistance in an electric network in the form of a grid graph between two diagonal 
vertices $[1,1]$ and $[k,n]$. We assume that each edge has resistance $1$ Ohm.
To obtain it, all we need  is, in addition for the number of spanning trees (that's the numerator),
the number of spanning forests $SF_k(n)$ of the graph $G_k(n)$ that have exactly two components, 
each component containing exactly one of the members of the pair   $\{[1,1],[k,n]\}$ (this is the denominator). 
The joint resistance is just the ratio.

In principle, we can apply the same method to obtain the generating function $S_k$. 
Empirically, we found that the denominator of $S_k$ is always the square of the denominator of $F_k$ times another polynomial $C_k$. 
Once the denominator is known, we can find the numerator in the same way as above.
So our focus is to find $C_k$.

The procedure {\tt DenomSFKN(k,t)} in the Maple package {\tt JointConductance.txt}, calculates $C_k$. For $2 \leq k \leq 4$,
we have
$$
C_2 = t-1 ,
$$
$$
C_3 = t^4-8t^3+17t^2-8t+1 ,
$$
$$
C_4 = t^{12}-46t^{11}+770t^{10}-6062t^9+24579t^8-55388t^7+72324t^6-55388t^5+24579t^4
$$
$$
-6062t^3+770t^2-46t+1 .
$$

{\bf Remark} 
By looking at the output of our Maple package, we conjectured that $R(k,n)$, the resistance between vertex $[1,1]$ and vertex $[k,n]$ in 
the $k \times n$ grid graph, $G_k(n)$, where each edge is a resistor of $1$ Ohm, 
is asymptotically $n/k$, for any fixed $k$, as $n \rightarrow \infty$. We proved it rigorously for $k \leq 6$, and
we wondered whether there is a human-generated ``electric proof'. Naturally we emailed Peter Doyle, the co-author of the
delightful masterpiece \cite{8}, who quickly came up with the following argument.

{\it 
Making the horizontal resistors into almost 
resistance-less gold wires gives the lower bound $R(k,n) \geq (n-1)/k$ since it is a parallel circuit of $k$ resistors of $n-1$ Ohms. 
For an upper bound of the same order, put 1 Ampere in at [1,1] and out at $[k,n]$, routing $1/k$ Ampere up each of the $k$ verticals. 
The energy dissipation is $k(n-1)/k^2+C(k) = (n-1)/k+C(k)$,
where the constant $C(k)$ is the energy dissipated along the top and bottom resistors. Specifically, $C(k) = 2(1-1/k)^2 + (1-2/k)^2 + \dots + (1/k)^2)$. So
$(n-1)/k \leq R(k,n) \leq (n-1)/k + C(k)$.}

We thank Peter Doyle for his kind permission to reproduce this  {\it electrifying} argument.

\section{The Statistic of the Number of Vertical Edges}

As mentioned in Section 2.4, often in enumerative combinatorics,  the class of interest has natural `statistics'.

In this section, we are interested in the statistic `the number of vertical edges',
defined on spanning trees of grid graphs. 
For  given $k$ and $n$, let, as above, $G_k(n)$
denote the $k \times n$ grid-graph.
Let $\mathcal{F}_{k,n}$ be its set of spanning trees.
if the weight is 1, then $\sum_{f \in \mathcal{F}_{k,n}} 1 =|\mathcal{F}_{k,n}|$ is the naive counting. Now let's define a natural statistic 

\centerline{$ver(T)$ = the number of vertical edges in the spanning tree $T$}

\noindent and the weight $w(T) =  v^{ver(T)}$, then the weighted counting follows:
$$
Ver_{k,n} (v) = \sum_{T \in \mathcal{F}_{k,n}} w(T)
$$
where $\mathcal{F}_{k,n}$ is the set of spanning trees of $G_k(n)$. 

We define the bivariate generating function
$$
g_{k}(v,t) = \sum_{n=0}^{\infty} Ver_{k,n} t^n.
$$
More generally, with our Maple package {\tt GFMatrix.txt}, and procedure {\tt VerGF}, 
we are able to obtain the bivariate generating function for an arbitrary graph of the form $G \times P_n$. 
The procedure {\tt VerGF} takes inputs $G$ (an arbitrary graph), $N$ (an integer determining how many data we use to 
find the recurrence relation) and two symbols $v$ and $t$.

The main tool for computing {\tt VerGF} is still the 
Matrix Tree Theorem and {\tt GuessRec}. 
But we need to modify  the Laplacian matrix for the graph. 
Instead of letting $a_{ij}=-1$ for $i \neq j$ and $\{i,j\} \in E(G \times P_n)$, 
we should consider whether the edge $\{i,j\}$ is a vertical edge. 
If so, we let $a_{i,j}=-v, a_{j,i}=-v$. The diagonal elements which are $(-1) \times$ 
(the sum of the rest entries on the same row) should change accordingly. 
The following theorems are for grid graphs when $2 \leq k \leq 4$ while $k=1$ is a trivial case because there are no vertical edges.

\begin{theorem}
The bivariate generating function for the weighted counting according to the number of vertical edges of spanning trees in $G_2(n)$ is:
$$
g_2(v,t) = \frac {vt}{1- \left( 2\,v+2 \right) t+{t}^{2}}  .
$$
\end{theorem}

\begin{theorem}
 The bivariate generating function for the weighted counting 
 according to the number of vertical edges  vertical edges of spanning trees in $G_3(n)$ is:
$$
g_3(v,t) = \frac {-{t}^{3}{v}^{2}+{v}^{2}t}{1- \left( 3\,{v}^{2}+8\,v+4 \right) t- \left( -10\,{v}^{2}-16\,v-6 \right) {t}^{2}- \left( 3\,{v}^{2}+8\,v+
4 \right) {t}^{3}+{t}^{4}} .
$$
\end{theorem}

\begin{theorem}
 The bivariate generating function for the weighted counting
 according to the number of vertical edges of spanning trees in $G_4(n)$ is:
$$
g_4(v,t) = \frac{numer(g_4)}{denom(g_4)}
$$
where
$$
numer(g_4) = {v}^{3}t+ \left( -16\,{v}^{5}-24\,{v}^{4}-9\,{v}^{3} \right) {t}^{3}+
 \left( 8\,{v}^{6}+40\,{v}^{5}+48\,{v}^{4}+16\,{v}^{3} \right) {t}^{4}
$$
$$
+ \left( -16\,{v}^{5}-24\,{v}^{4}-9\,{v}^{3} \right) {t}^{5}+{v}^{3}{t
}^{7}
$$
and
$$
denom(g_4) = 1- \left( 4\,{v}^{3}+20\,{v}^{2}+24\,v+8 \right) t- \left( -52\,{v}^{4
}-192\,{v}^{3}-256\,{v}^{2}-144\,v-28 \right) {t}^{2}
$$
$$
- \left( 64\,{v}^
{5}+416\,{v}^{4}+892\,{v}^{3}+844\,{v}^{2}+360\,v+56 \right) {t}^{3}
$$
$$
- \left( -16\,{v}^{6}-160\,{v}^{5}-744\,{v}^{4}-1408\,{v}^{3}-1216\,{v}
^{2}-480\,v-70 \right) {t}^{4}
$$
$$
- \left( 64\,{v}^{5}+416\,{v}^{4}+892\,{
v}^{3}+844\,{v}^{2}+360\,v+56 \right) {t}^{5}
$$
$$
- \left( -52\,{v}^{4}-192
\,{v}^{3}-256\,{v}^{2}-144\,v-28 \right) {t}^{6}
- \left( 4\,{v}^{3}+20
\,{v}^{2}+24\,v+8 \right) {t}^{7}+{t}^{8}  .
$$
\end{theorem}

With the Maple package {\tt BiVariateMoms.txt} and its {\tt Story} procedure from

{\tt http://sites.math.rutgers.edu/\~{}zeilberg/tokhniot/BiVariateMoms.txt}, 

\noindent the expectation, variance and higher moments can be easily analyzed. We calculated up to the 4th moment for $G_2(n)$. 
For $k=3,4$, you can find the output files from

{\tt http://sites.math.rutgers.edu/\~{}yao/OutputStatisticVerticalk=3.txt},

{\tt http://sites.math.rutgers.edu/\~{}yao/OutputStatisticVerticalk=4.txt}.

\begin{theorem}
 The moments of the statistic: the number of vertical edges in the spanning trees of $G_2(n)$ are as follows:

Let $b$ be the largest positive root of the polynomial equation
$$
b^2-4b+1 = 0
$$
whose floating-point approximation is 3.732050808, then the size of the $n$-th family (i.e. straight enumeration) is very close to
$$
{\frac {{b}^{n+1}}{-2+4\,b}} .
$$
The average of the statistics is, asymptotically 
$$
\frac{1}{3}+\frac{1}{3}\,{\frac { \left( -1+2\,b \right) n}{b}}  .
$$
The variance of the statistics is, asymptotically 
$$
-\frac{1}{9}+\frac{1}{9}\,{\frac { \left( 7\,b-2 \right) n}{-1+4\,b}} .
$$
The skewness of the statistics is, asymptotically 
$$
{\frac {780\,b-209}{ \left( 4053\,b-1086 \right) {n}^{3}+ \left( -7020
\,b+1881 \right) {n}^{2}+ \left( 4053\,b-1086 \right) n-780\,b+209}}.
$$
The kurtosis of the statistics is, asymptotically 
$$
3\,{\frac { \left( 32592\,b-8733 \right) {n}^{2}+ \left( -56451\,b+
15126 \right) n+21728\,b-5822}{ \left( 32592\,b-8733 \right) {n}^{2}+
 \left( -37634\,b+10084 \right) n+10864\,b-2911}} .
$$
\end{theorem}

\section{Application of the $C$-finite Ansatz to  Almost-Diagonal Matrices}

So far, we have seen applications of the $C$-finite ansatz methodology 
for automatically computing generating functions for
enumerating spanning trees/forests for certain infinite families of graphs. 

The second case study is completely different, and in a sense more general,
since the former framework may be subsumed in this new context.

\begin{definition}
Diagonal matrices $A$ are square matrices in which the entries outside the main diagonal are $0$, i.e. $a_{ij} = 0$ if $i \neq j$.
\end{definition}

\begin{definition}
 An almost-diagonal matrix $A$ is a square matrices 
in which $a_{i,j} = 0$ if $j-i \geq k_1$ or $i-j \geq k_2$ for some fixed positive integers $k_1, k_2$ 
and $\forall i_1, j_1, i_2, j_2$, if $i_1-j_1 = i_2-j_2$, then $a_{i_1 j_1} = a_{i_2 j_2}$. 
\end{definition}

For simplicity, we use the notation $L=$[$n$, [the first $k_1$ entries in the first row], [the first $k_2$ entries in the first column]] to denote 
the $n \times n$ matrix with these specifications.
Note that this notation already contains all information we need to reconstruct this matrix. 
For example, [6, [1,2,3], [1,4]] is the matrix
$$
\begin{bmatrix}
 1 & 2 & 3 & 0 & 0 & 0 \\
 4 & 1 & 2 & 3 & 0 & 0 \\
 0 & 4 & 1 & 2 & 3 & 0 \\
 0 & 0 & 4 & 1 & 2 & 3 \\
 0 & 0 & 0 & 4 & 1 & 2 \\
 0 & 0 & 0 & 0 & 4 & 1
\end{bmatrix} .
$$

The following is the Maple procedure {\tt DiagMatrixL}  (in our Maple package {\tt GFMatrix.txt}),
which inputs such a list $L$ and outputs the corresponding matrix. 
\vspace{1mm}\noindent
{\obeylines
{\tt
DiagMatrixL:=proc(L) local n, r1, c1,p,q,S,M,i:
n:=L[1]:
r1:=L[2]:
c1:=L[3]:
p:=nops(r1)-1:
q:=nops(c1)-1:
if r1[1] <> c1[1] then
  return fail:
fi:
S:=[0\$(n-1-q), seq(c1[q-i+1],i=0..q-1), op(r1), 0\$(n-1-p)]:
M:=[0\$n]:
for i from 1 to n do
  M[i]:=[op(max(0,n-1-q)+q+2-i..max(0,n-1-q)+q+1+n-i,S)]:
od:
return M:
end:
}
}
\vspace{1mm}\noindent
For this matrix, $k_1=3$ and $k_2=2$. 
Let $k_1, k_2$ be fixed and $M_1, M_2$ be two lists of numbers or symbols of length $k_1$ and $k_2$ respectively, 
$A_k$ is the almost-diagonal matrix represented by the list $L_k = [k, M_1, M_2]$. Note that the first elements in the lists $M_1$ and $M_2$ must be identical. 

Having fixed two lists $M_1$ of length $k_1$ and $M_2$ of length $k_2$, (where $M_1[1]=M_2[1]$), it is of interest
to derive {\it automatically}, the generating function (that is always a rational function for reasons that will
soon become clear), $\sum_{k=0}^{\infty} a_k \, t^k$, where
 $a_k$ denotes the determinant of the $k \times k$ almost-diagonal matrix whose first row starts with $M_1$, and first column
starts with $M_2$. Analogously, it is also of interest to do the analogous problem when the determinant is replaced by the permanent.

Here is the Maple procedure {\tt GFfamilyDet} which takes inputs (i) $A$: a name of a Maple procedure that inputs an integer $n$ 
and outputs an $n \times n$ matrix according to some rule, e.g., the almost-diagonal matrices, 
(ii) a variable name $t$, (iii) two integers $m$ and $n$ which are the lower and upper bounds 
of the sequence of determinants we consider. It outputs a rational function in $t$, say $R(t)$, which is the generating function of the sequence.
\vspace{1mm}\noindent
{\obeylines
{\tt
GFfamilyDet:=proc(A,t,m,n) local i,rec,GF,B,gu,Denom,L,Numer:
L:=[seq(det(A(i)),i=1..n)]:
rec:=GuessRec([op(m..n,L)])[2]:
gu:=solve(B-1-add(t**i*rec[i]*B,i=1..nops(rec)), {B}):
Denom:=denom(subs(gu,B)):
Numer:=Denom*(1+add(L[i]*t**i, i=1..n)):
Numer:=add(coeff(Numer,t,i)*t**i, i=0..degree(Denom,t)):
Numer/Denom:
end: 
}
}
\vspace{1mm}\noindent
Similarly we have procedure {\tt GFfamilyPer} for the permanent. Let's look at an example. 
The following is a sample procedure which considers the family of 
almost diagonal matrices which the first row $[2,3]$ and the first column $[2,4,5]$. 
\vspace{1mm}\noindent
{\obeylines
{\tt
SampleB:=proc(n) local L,M:
L:=[n, [2,3], [2,4,5]]:
M:=DiagMatrixL(L):
end:
}
}
Then {\tt GFfamilyDet(SampleB, t, 10, 50)} will return the generating function 
$$
-\frac{1}{ 45\,{t}^{3}-12\,{t}^{2}+2\,t-1 } .
$$

It turns out, that for this problem, the more `conceptual' approach of setting up a transfer matrix
also works well. But don't worry, the computer can do the `research' all by itself, with only
a minimum amount of human pre-processing.

We will now describe this more conceptual approach, that may be called {\it symbolic dynamical programming}, where
the computer sets up, {\it automatically}, a finite-state scheme, by {\it dynamically} discovering the
set of states, and automatically figures out the transfer matrix.

\section{The Symbolic Dynamic Programming Approach}

Recall from Linear Algebra 101, 

\begin{theorem}[Cofactor Expansion] Let $|A|$ denote the determinant of an $n \times n$ matrix $A$, then
$$
|A| = \sum_{j=1}^{n} (-1)^{i+j} a_{ij} M_{ij}, \quad \forall i \in [n],
$$
where $M_{ij}$ is the $(i,j)$-minor.
\end{theorem}

We'd like to consider the Cofactor Expansion for almost-diagonal matrices along the first row. 
For simplicity, we assume while $a_{i,j} = 0$ if $j-i \geq k_1$ or $i-j \geq k_2$ 
for some fixed positive integers $k_1, k_2$, and if $-k_2 < j_1-i_1 < j_2-i_2 < k_1$, then $a_{i_1 j_1} \neq a_{i_2 j_2}$. 
Under this assumption, for any minors we obtain through recursive Cofactor Expansion along the first row, 
the dimension, the first row and the first column should provide enough information to reconstruct the matrix. 

For an almost-diagonal matrix represented by $L=$[Dimension, [the first $k_1$ entries in the first row], 
[the first $k_2$ entries in the first column]], any minor can be represented by 
[Dimension, [entries in the first row up to the last nonzero entry], [entries in the first column up to the last nonzero entry]].

Our goal in this section is the same as the last one, to get a generating function for the determinant or permanent of almost-diagonal matrices $A_k$ with dimension $k$. Once we have those almost-diagonal matrices, the first step is to do a one-step expansion as follows:
\vspace{1mm}\noindent
{\obeylines
{\tt
ExpandMatrixL:=proc(L,L1) 
local n,R,C,dim,R1,C1,i,r,S,candidate,newrow,newcol,gu,mu,temp,p,q,j:
n:=L[1]:
R:=L[2]:
C:=L[3]:
p:=nops(R)-1:
q:=nops(C)-1:
dim:=L1[1]:
R1:=L1[2]:
C1:=L1[3]:

if R1=[] or C1=[] then
  return {}:
elif R[1]<>C[1] or R1[1]<>C1[1] or dim>n then 
  return fail:
else

  S:=\{\}:

  gu:=[0\$(n-1-q), seq(C[q-i+1],i=0..q-1), op(R), 0\$(n-1-p)]:
  candidate:=[0\$nops(R1),R1[-1]]:
  for i from 1 to nops(R1) do
    mu:=R1[i]:
    for j from n-q to nops(gu) do
      if gu[j]=mu then
        candidate[i]:=gu[j-1]:
      fi:
    od:
  od:
      
  for i from n-q to nops(gu) do
    if gu[i] = R1[2] then
      temp:=i:
      break:
    fi:
  od:

  for i from 1 to nops(R1) do
    if i = 1 then
      mu:=[R1[i]*(-1)**(i+1), [dim-1,[op(i+1..nops(candidate), candidate)],
      [seq(gu[temp-i],i=1..temp-n+q)]]]:
      S:=S union {mu}:
    else
      mu:=[R1[i]*(-1)**(i+1), [dim-1, [op(1..i-1, candidate),
      op(i+1..nops(candidate), candidate)], [op(2..nops(C1), C1)]]]:
      S:=S union {mu}:
    fi:
  od:

  return S:

fi:

end:
}
}
\vspace{1mm}\noindent
The {\tt ExpandMatrixL} procedure inputs a data structure $L =$ [Dimension, first\_row=[ ], first\_col=[ ]] 
as the matrix we start and the other data structure $L1$ as the current minor we have, 
expands $L1$ along its first row and outputs a list of [multiplicity, data structure]. 

We would like to generate all the "children" of an almost-diagonal matrix regardless of the dimension, i.e., 
two lists $L$ represent the same child as long as their first\_rows and first\_columns are the same, respectively. 
The set of "children" is the scheme of the almost diagonal matrices in this case.

The following is the Maple procedure {\tt ChildrenMatrixL} which inputs a data structure $L$ and 
outputs the set of its "children" under Cofactor Expansion along the first row:
\vspace{1mm}\noindent
{\obeylines
{\tt
ChildrenMatrixL:=proc(L) local S,t,T,dim,U,u,s:
dim:=L[1]:
S:=\{[op(2..3,L)]\}:
T:=\{seq([op(2..3,t[2])],t in ExpandMatrixL(L,L))\}:
while T minus S <> \{\} do
  U:=T minus S:
  S:=S union T:
  T:=\{\}:
  for u in U do
    T:=T union \{seq([op(2..3,t[2])],t in ExpandMatrixL(L,[dim,op(u)]))\}:
  od:
od:
for s in S do
  if s[1]=[] or s[2]=[] then
    S:=S minus \{s\}:
  fi:
od:
S:
end:
}
}
\vspace{1mm}\noindent
After we have the scheme $S$, by the Cofactor Expansion of any element in the scheme, a system of algebraic equations follows. 
For children in $S$, it's convenient to let the almost-diagonal matrix be the first one $C_1$ and for 
the rest, any arbitrary ordering will do. For example, if after Cofactor Expansion for $C_1$, $c_2$ "copies" of 
$C_2$ and $c_3$ "copies" of $C_3$ are obtained, then the equation will be 
$$
C_1 = 1+ c_2 t C_2 + c_3 t C_3  .
$$
However, if the above equation is for $C_i, i \neq 1$, i.e. $C_i$ is not the almost-diagonal matrix itself, then the equation will be slightly different:
$$
C_i = c_2 t C_2 + c_3 t C_3 .
$$
Here $t$ is a symbol as we assume the generating function is a rational function of $t$.

Here is the Maple code that implements how we get the generating function for the determinant of a family of 
almost-diagonal matrices by solving a system of algebraic equations:
\vspace{1mm}\noindent
{\obeylines
{\tt
GFMatrixL:=proc(L,t) local S,dim,var,eq,n,A,i,result,gu,mu:
dim:=L[1]:
S:=ChildrenMatrixL(L):
S:=[[op(2..3,L)], op(S minus \{[op(2..3,L)]\})]:
n:=nops(S):
var:=\{seq(A[i],i=1..n)\}:
eq:=\{\}:
for i from 1 to 1 do
  result:=ExpandMatrixL(L,[dim,op(S[i])]):
  for gu in result do
  if gu[2][2]=[] or gu[2][3]=[] then
    result:=result minus \{gu\}:
  fi:
  od:
  eq:=eq union \{A[i] - 1 - add(gu[1]*t*A[CountRank(S, 
  [op(2..3, gu[2])])], gu in result)\}:
od:
for i from 2 to n do
  result:=ExpandMatrixL(L,[dim,op(S[i])]):
  for gu in result do
  if gu[2][2]=[] or gu[2][3]=[] then
    result:=result minus {gu}:
  fi:
  od:
  eq:=eq union \{A[i] - add(gu[1]*t*A[CountRank(S, [op(2..3, gu[2])])], 
  gu in result)\}:
od:
gu:=solve(eq, var)[1]:
subs(gu, A[1]):

end:

}
}
\vspace{1mm}\noindent
{\tt GFMatrixL([20, [2, 3], [2, 4, 5]], t)} returns
$$
- \frac{1}{ 45\,{t}^{3}-12\,{t}^{2}+2\,t-1}  .
$$
Compared to empirical approach, the `symbolic dynamical programming' method is faster and more efficient
for the moderate-size examples that we tried out.
However, as the lists will grow larger, it is likely that the former method will win out,
since with this non-guessing approach, it is equally fast to get generating functions for
determinants and permanents, and as we all know, permanents are hard.

The advantage of the present method  is that it is more appealing to humans, and does not
require any `meta-level' act of faith.
However, both methods are very versatile and are great experimental approaches for enumerative combinatorics problems. 
We hope that our readers will find other applications.

\section{Remarks}
Rather than trying to tackle each enumeration problem, one at a time, using
ad hoc human ingenuity each time, building up an intricate transfer matrix, and
only using the computer at the end as a symbolic calculator, it is a much better
use of our beloved silicon servants
(soon to become our masters!) to replace `thinking' by `meta-thinking', i.e. develop experimental mathematics methods
that can handle many different types of problems. In the two case studies discussed here,
every thing was made rigorous, but if one can make semi-rigorous and even non-rigorous
discoveries, as long as they are {\it interesting}, one should not be hung up
on rigorous proofs. In other words, if you can find a rigorous justification (like in these
two case studies) that's nice, but if you can't, that's also nice!

\chapter{Analysis of Quicksort Algorithms}
This chapter is adapted from \cite{49}, which has been published on {\it Journal of Difference Equations and Applications}. It is also available on arXiv.org, number 1905.00118.

\section{Introduction}

A sorting algorithm is an algorithm that rearranges elements of a list in a certain order, the most frequently used orders being numerical order and lexicographical order. Sorting algorithms play a significant role in computer science since efficient sorting is important for optimizing the efficiency of other algorithms which require input data to be in sorted lists. In this chapter, our focus is {\it Quicksort}. 

Quicksort was developed by British computer scientist Tony Hoare in 1959 and published in 1961. It has been a commonly used algorithm for sorting since then and is still widely used in industry. 

The main idea for Quicksort is that we choose a pivot randomly and then compare the other elements with the pivot, smaller elements being placed on the left side of the pivot and larger elements on the right side of the pivot. Then we recursively apply the same operation to the sublists obtained from the partition step. As for the specific implementations, there can be numerous variants, some of which are at least interesting from a theoretical perspective despite their rare use in the real world.

It is well-known that the worst-case performance of Quicksort is $O(n^2)$ and the average performance is $O(n \log n)$. However, we are also interested in the explicit closed-form expressions for the moments of Quicksort's performance, i.e., running time, in terms of the number of comparisons and/or the number of swaps. In this chapter, only lists or arrays containing distinct elements are considered.

The chapter is organized as follows. In Section 4.2, we review related work on the number of comparisons of 1-pivot Quicksort, whose methodology is essential for further study. In Section 4.3, the numbers of swaps of several variants of 1-pivot Quicksort are considered. In Section 4.4, we extend our study to multi-pivot Quicksort. In Section 4.5, the technique to obtain more moments and the scaled limiting distribution are discussed. In the last section we discuss some potential improvements for Quicksort, summarize the main results of this chapter and make final remarks on the methodology of experimental mathematics. 

\section{Related Work}
In the masterpiece of Shalosh B. Ekhad and Doron Zeilberger \cite{12}, they managed to find the explicit expressions for expectation, variance and higher moments of the number of comparisons of 1-pivot Quicksort with an experimental mathematics approach, which is also considered as some form of
``machine learning." Here we will review the results they discovered or rediscovered. 

Let $C_n$ be the random variable  ``number of comparisons in Quicksort applied to lists of length $n$," $n \geq 0$.

\begin{theorem}[\cite{24}, p.8, end of section 1.3; \cite{16}, Eq. (2.14), p. 29, and other places]
$$
E[C_n]=2(n+1) H_1(n) - 4n .
$$
Here $H_1(n)$ are the {\it Harmonic numbers}
$$
H_1(n):=\sum_{i=1}^{n} \frac{1}{i} .
$$
\end{theorem}

More generally, in following theorems, we introduce the notation 
$$
H_k(n) := \sum_{i=1}^n \frac{1}{i^k}.
$$

\begin{theorem}[Knuth, \cite{26}, answer to Ex. 8(b) in section  6.2.2)]
$$
var[C_n] = n ( 7\,n+13 )  \, - \, 2\,(n+1)\, H_{{1}} ( n )  -4\, ( n+1 ) ^{2}H_{{2}} ( n )  .
$$
Its asymptotic expression is
$$
( 7 \,- \,\frac{2}{3}\,{\pi }^{2} ) {n}^{2}+ ( 13-2\,\ln  ( n
 ) -2\,\gamma-4/3\,{\pi }^{2} ) n-2\,\ln  ( n ) 
-2\,\gamma-2/3\,{\pi }^{2} \, + \, o(1)  .       
$$
\end{theorem}

\begin{theorem}[Zeilberger, \cite{12}] The third moment about the mean of $C_n$ is 
$$
-n ( 19\,{n}^{2}+81\,n+104 ) +H_{{1}} ( n ) 
 ( 14\,n+14 ) +12\, ( n+1 ) ^{2}H_{{2}} ( n
 ) +16\, ( n+1 ) ^{3}H_{{3}} ( n )  .
$$
It is asymptotic  to
$$
( -19+16\,\zeta  ( 3 )  ) {n}^{3}+ ( -81+2
\,{\pi }^{2}+48\,\zeta  ( 3 )  ) {n}^{2}+ ( -104+
14\,\ln  ( n ) 
+14\,\gamma+4\,{\pi }^{2}
$$
$$
+48\,\zeta  ( 3
 )  ) n 
+14\,\ln  ( n ) +14\,\gamma+2\,{\pi }^{2}
+16\,\zeta  ( 3 ) \, + \, o(1)  .
$$
It follows that the limit of the scaled third moment (skewness) converges to
$$
{\frac{-19+16\,\zeta  ( 3 ) }{ ( 7-2/3\,{\pi }^{2} ) ^{3/2}}} \, = \, 0.8548818671325885 \dots \quad .
$$
\end{theorem}

\begin{theorem} [Zeilberger, \cite{12}] The fourth moment about the mean of $C_n$ is 
$$
\frac{1}{9} \,n ( 2260\,{n}^{3}+9658\,{n}^{2}+15497\,n+11357 ) -2\,
 ( n+1 )  ( 42\,{n}^{2}+78\,n+77 ) H_{{1}}
 ( n ) 
$$
$$
+12\, ( n+1 ) ^{2} ( H_{{1}} ( n
 )  ) ^{2}+ ( -4\, ( 42\,{n}^{2}+78\,n+31
 )  ( n+1 ) ^{2}+48\, ( n+1 ) ^{3}H_{{1}}
 ( n )  ) H_{{2}} ( n ) 
$$
$$
+48\, ( n+1
 ) ^{4} ( H_{{2}} ( n )  ) ^{2}-96\,
 ( n+1 ) ^{3}H_{{3}} ( n ) -96\, ( n+1
 ) ^{4}H_{{4}} ( n ) .
$$
It is asymptotic  to
$$
 ( {\frac {2260}{9}}-28\,{\pi }^{2}+{\frac {4}{15}}\,{\pi }^{4}
 ) {n}^{4}+ ( {\frac {9658}{9}}-84\,\ln  ( n ) -
84\,\gamma+1/6\, ( -648+48\,\ln  ( n ) +48\,\gamma
 ) {\pi }^{2}
$$
$$
+{\frac {16}{15}}\,{\pi }^{4}-96\,\zeta  ( 3
 )  ) {n}^{3}
$$
$$
+ ( {\frac {15497}{9}}-240\,\ln  ( n
 ) -240\,\gamma+12\, ( \ln  ( n ) +\gamma
 ) ^{2}+1/6\, ( -916+144\,\ln  ( n ) +144\,\gamma
 ) {\pi }^{2}
$$
$$
+8/5\,{\pi }^{4}-288\,\zeta  ( 3 ) 
 ) {n}^{2}
$$
$$
+ ( {\frac {11357}{9}}-310\,\ln  ( n ) 
-310\,\gamma+24\, ( \ln  ( n ) +\gamma ) ^{2}+1/6
\, ( -560+144\,\ln  ( n ) +144\,\gamma ) {\pi }^{
2}
$$
$$
+{\frac {16}{15}}\,{\pi }^{4}-288\,\zeta  ( 3 )  ) n
$$
$$
-154\,\ln  ( n ) -154\,\gamma+12\, ( \ln  ( n
 ) +\gamma ) ^{2}+1/6\, ( -124+48\,\ln  ( n
 ) +48\,\gamma ) {\pi }^{2}
$$
$$
+{\frac {4}{15}}\,{\pi }^{4}-96
\,\zeta  ( 3 ) \, + \, o(1) .
$$
It follows that the limit of the scaled fourth moment (kurtosis) converges to
$$
{\frac{{\frac{2260}{9}}-28\,{\pi }^{2}+{\frac{4}{15}}\,{\pi }^{4}}{
 ( 7-2/3\,{\pi }^{2} ) ^{2}}}
\, = \,
4.1781156382698542\dots \quad .
$$
\end{theorem}

Results for higher moments, more precisely, up to the eighth moment, are also discovered and discussed by Shalosh B. Ekhad and Doron Zeilberger in \cite{12}.

Before this article, there are already human approaches to find the expectation and variance for the number of comparisons. Let $c_n = E[C_n]$. Since the pivot can be the $k$-th smallest element in the list $(k=1,2,\dots,n)$, we have the recurrence relation
$$
c_n = 
\frac{1}{n} \sum_{k=1}^{n} ((n-1)+ c_{k-1}+ c_{n-k} ) = (n-1) + \frac{1}{n} \sum_{k=1}^{n} (c_{k-1}+c_{n-k}) 
= (n-1) + \frac{2}{n} \sum_{k=1}^{n} \, c_{k-1},
$$
because the expected number of comparisons for the sublist before the pivot is $c_{k-1}$ and that for the sublist after the pivot is $c_{n-k}$. From this recurrence relation, complicated human-generated manipulatorics is needed to rigorously derive the closed form. For the variance, the calculation is much more complicated. For higher moments, we doubt that human approach is realistic. 

The experimental mathematics approach is more straightforward and more powerful. For the expectation, a list of data can be obtained through the recurrence relation and the initial condition. Then with an educated guess that $c_n$ is a polynomial of degree one in both $n$ and $H_1(n)$, i.e.,
$$
c_n = a + bn + cH_1(n) + dnH_1(n)
$$
where $a,b,c,d$ are undetermined coefficients, we can solve for these coefficients by plugging sufficiently many $n$ and $c_n$ in this equation.

For higher moments, there is a similar recurrence relation for the probability generating function of $C_n$. With the probability generating function, a list of data of any fixed moment can be obtained. Then with another appropriate educated guess of the form of the higher moments, the explicit expression follows.  

In \cite{12}, it is already discussed that this experimental mathematics approach, which utilizes a recurrence relation to study the Quicksort algorithms, is actually rigorous by pointing out that this is a finite calculation and by referring to results in \cite{35} and \cite{36}.

\section{Number of Swaps of 1-Pivot Quicksort}
The performance of Quicksort depends on the number of swaps and comparisons performed. In reality, a swap usually takes more computing resources than a comparison. The difficulty in studying the number of swaps is that the number of swaps depends on how we implement the Quicksort algorithm while the number of comparisons are the same despite the specific implementations. 

Since only the number of comparisons is considered in \cite{12}, the Quicksort model in \cite{12} is that one picks the pivot randomly, compares each non-pivot element with the pivot and then places them in one of the two new lists $L_1$ and $L_2$ where the former contains all elements smaller than the pivot and the latter contains those greater than the pivot. Under this model there is no swap, but a lot of memory is needed. For convenience, let's call this model Variant Nulla. 

In this section, we consider the random variable, the number of swaps $X_n$, in different Quicksort variants. Some of them may not be efficient or widely used in industry; however, we treat them as an interesting problem and model in permutations and discrete mathematics. In the first subsection, we also demonstrate our experimental mathematics approaches step by step. 

\subsection{Variant I}
The first variant is that we always choose the first (or equivalently, the last) element in the list of length $n$ as the pivot, then we compare the other elements with the pivot. We compare the second element with the pivot first. If it is greater than the pivot, it stays where it is, otherwise we remove it from the list and then insert it before the pivot. Though this is somewhat different from the ``traditional swap," we define this operation as a swap. Generally, every time we find an element smaller than the pivot, we insert it before the pivot.

Hence, after $n-1$ comparisons and some number of swaps, the partition is achieved, i.e., all elements on the left of the pivot are smaller than the pivot and all elements on the right of the pivot are greater than the pivot. The difference between this variant and Variant Nulla is that this one does not need to create new lists so that it saves memory. 

Let $P_n(t)$ be the probability generating function for the number of swaps, i.e.,
$$
P_n(t) = \sum_{k=0}^{\infty}  P(X_n = k) \, t^k,
$$
where for only finitely many integers $k$, we have that $P(X_n = k)$ is nonzero.

We have the recurrence relation 
$$
P_n(t) = \frac{1}{n} \sum_{k=1}^n P_{k-1}(t) P_{n-k}(t) t^{k-1},
$$
with the initial condition $P_0(t)=P_1(t)=1$ because for any fixed $k \in \{1,2,\dots,n\}$, the probability that the pivot is the $k$-th smallest is $\frac{1}{n}$ and there are exactly $k-1$ swaps when the pivot is the $k$-th smallest.

The Maple procedure {\tt SwapPQs(n,t)} in the package {\tt Quicksort.txt} implements the recurrence of the probability generating function. 

Recall that the $r$-th moment is given in terms of the probability generating function
$$
E[X_n^r]= (t\frac{d}{dt})^r P_n(t) \, \vert_{t=1} .
$$
The moment about the mean
$$
m_r(n) := E[(X_n - c_n)^r] ,
$$
can be easily derived from the raw moments $\{ E[X_n^l] \,|\, 1 \leq l \leq r\}$, using the binomial theorem and
linearity of expectation. Another way to get the moment about the mean is by considering
$$
m_r(n)= (t\frac{d}{dt})^r (\frac{P_n(t)}{t^{c_n}}) \, \vert_{t=1} .
$$
Recall that 
$$
H_k(n) := \sum_{i=1}^n \frac{1}{i^k}.
$$
Our educated guess is that there exists a polynomial of $r+1$ variables $F_r(x_0, x_1, \dots, x_r)$ such that 
$$
m_r(n) = F_r(n, H_1(n), \dots, H_r(n)).
$$
With the Maple procedure {\tt QsMFn}, we can easily obtain the following theorems by just entering {\tt QsMFn(SwapPQs, t, n, Hn, r)} where $r$ represents the moment you are interested in. When $r=1$, it returns the explicit expression for its mean rather than the trivial ``first moment about the mean''.  

\begin{theorem}
The expectation of the number of swaps of Quicksort for a list of length $n$ under Variant I is
$$
E[X_n] = (n+1) H_1(n) - 2n.
$$
\end{theorem}

\begin{theorem}
The variance of $X_n$ is 
$$
2n(n+2) - (n+1) H_1(n) - (n+1)^2 H_2(n).
$$
\end{theorem}

\begin{theorem}
The third moment about the mean of $X_n$ is 
$$
-\frac{9}{4} n (n+3)^2 + (4n+4) H_1(n) + 3(n+1)^2 H_2(n) + 2(n+1)^3 H_3(n).
$$
\end{theorem}

\begin{theorem}
The fourth moment about the mean of $X_n$ is 
$$
\frac{1}{18} n (335n^3 + 1568n^2 + 3067n + 2770) - 3(n+1)(4n^2+8n+9)H_1(n) +3(n+1)^2 H_1(n)^2
$$
$$
+(-(12n^2+24n+19)(n+1)^2 + 6(n+1)^3 H_1(n) ) H_2(n) + 3(n+1)^4 H_2(n)^2
$$
$$
- 12(n+1)^3 H_3(n) - 6(n+1)^4 H_4(n).
$$
\end{theorem}

The explicit expressions for higher moments can be easily calculated automatically with the Maple procedure {\tt QsMFn} and the interested readers are encouraged to find those formulas on their own. 

\subsection{Variant II}
The second variant is similar to the first one. One tiny difference is that instead of choosing the first or last element as the pivot, the index of the pivot is chosen uniformly at random. For example, we choose the $i$-th element, which is the $k$-th smallest, as the pivot. Then we compare those non-pivot elements with the pivot. If $i\neq 1$, the first element will be compared with the pivot first. If it is smaller than the pivot, it stays there, otherwise it is moved to the end of the list. After comparing all the left-side elements with the pivot, we look at those elements whose indexes are originally greater than $i$. If they are greater than the pivot, no swap occurs; otherwise insert them before the pivot. 

In this case, the recurrence of the probability generating function $P_n(t)$ is more complicated as a consequence of that the number of swaps given that $i$ and $k$ is known is still a random variable rather than a fixed number as the case in Variant I. 

Let $Q(n,k,i,t)$ be the probability generating function for such a random variable. In fact, given a random permutation in the permutation group $S_n$ and that the $i$-th element is $k$, the number of swaps equals to the number of elements which are before $k$ and greater than $k$ or after $k$ and smaller than $k$. Hence, if there are $j$ elements which are before $k$ and smaller than $k$, then there are $i-1-j$ elements which are before $k$ and greater than $k$ and there are $k-1-j$ elements which are after $k$ and smaller than $k$. So in this case the number of swaps is $i+k-2-2j$. 

Then we need to determine the range of $j$. Obviously it is at least 0. In total there are $k-1$ elements which are less than $k$, at most $n-i$ of them occurring after $k$, so $j \geq k-1-n+i$. As for the upper bound, since there are only $i-1$ elements before $k$, we have $j \leq i-1$. Evidently, $j \leq k-1$ as well. Therefore the range of $j$ is $[\,\max(k-1-n+i, 0),\, \min(i-1, k-1)\,]$.

As for the probability that there are exactly $j$ elements which are before $k$ and smaller than $k$, it equals to 
$$
\binom{i-1}{j} \prod_{s=0}^{j-1} \frac{k-1-s}{n-1-s} \prod_{s=0}^{i-j-2} \frac{n-k-s}{n-1-j-s}.
$$
Consequently, the probability generating function is
$$
Q(n,k,i,t) = \sum_{j=\max(k-1-n+i, 0)}^{\min(i-1, k-1)}  \binom{i-1}{j} \prod_{s=0}^{j-1} \frac{k-1-s}{n-1-s} \prod_{s=0}^{i-j-2} \frac{n-k-s}{n-1-j-s} t^{i+k-2-2j},
$$
which is implemented by the Maple procedure {\tt PerProb(n, k, i, t)}. For example, {\tt PerProb(9, 5, 5, t)} returns 
$$
\frac{1}{70} t^8 + \frac{8}{35} t^6 + \frac{18}{35} t^4 + \frac{8}{35} t^2 + \frac{1}{70}.
$$
We have the recurrence relation
$$
P_n(t) = \frac{1}{n^2} \sum_{k=1}^n  \sum_{i=1}^n P_{k-1}(t) P_{n-k}(t)  Q(n,k,i,t),
$$
with the initial condition $P_0(t)=P_1(t)=1$, which is implemented by the Maple procedure {\tt SwapPQ(n, t)}. The following theorems follow immediately.

\begin{theorem}
The expectation of the number of swaps of Quicksort for a list of length $n$ under Variant II is
$$
E[X_n] = (n+1) H_1(n) - 2n.
$$
\end{theorem}

\begin{theorem}
The variance of $X_n$ is 
$$
\frac{1}{6} n(11n+17) - \frac{1}{3}(n+1) H_1(n) - (n+1)^2 H_2(n).
$$
\end{theorem}

\begin{theorem}
The third moment about the mean of $X_n$ is 
$$
-\frac{1}{6} n (14n^2+57n+73) + (2n+2) H_1(n) + (n+1)^2 H_2(n) + 2(n+1)^3 H_3(n).
$$
\end{theorem}

\begin{theorem}
The fourth moment about the mean of $X_n$ is 
$$
\frac{1}{90} n (1496n^3 + 5531n^2 + 8527n + 6922) - \frac{1}{15}(n+1)(55n^2+85n+173)H_1(n)
$$
$$
+\frac{1}{3}(n+1)^2 H_1(n)^2
+(-\frac{1}{3}(33n^2+51n+25)(n+1)^2 + 2(n+1)^3 H_1(n) ) H_2(n) 
$$
$$
+ 3(n+1)^4 H_2(n)^2
-4(n+1)^3 H_3(n) - 6(n+1)^4 H_4(n).
$$
\end{theorem}

Higher moments can also be easily obtained by entering {\tt QsMFn(SwapPQ, t, n, Hn, r)} where $r$ represents the $r$-th moment you are interested in. 

Comparing with Variant I where the index of the pivot is fixed, we find that these two variants have the same expected number of swaps. However, the variance and actually all even moments of the second variant are smaller. Considering that the average performance is already $O(n \log n)$ which is not far from the best scenario, it is favorable that a Quicksort algorithm has smaller variance. In conclusion, for this model, a randomly-chosen-index pivot can improve the performance of the algorithm. 

\subsection{Variant III}
Next we'd like to study the most widely used in-place Quicksort. This algorithm is called Lomuto partition scheme, which is attributed to Nico Lomuto and popularized by Bentley in his book {\it Programming Pearls} and Cormen, {\it et al.} in their book {\it Introduction to Algorithms}. This scheme chooses a pivot that is typically the last element in the list. Two indexes, $i$, the insertion index, and $j$, the search index are maintained. The following is the pseudo code for this variant. 
{\obeylines
{\tt 
{\bf algorithm} quicksort(A, s, t) {\bf is}
   \quad {\bf if} s < t {\bf then}
     \qquad  p := partition(A, s, t)
     \qquad   quicksort(A, s, p - 1)
     \qquad quicksort(A, p + 1, t)
}
} 

{\obeylines
{\tt 
{\bf algorithm} partition(A, s, t) {\bf is}
  \quad   pivot := A[t]
  \quad  i := s
   \quad {\bf for} j := s {\bf to} t - 1 {\bf do}
    \qquad    {\bf if} A[j] < pivot {\bf then}
    \qquad \quad        swap A[i] with A[j]
        \qquad \quad    i := i + 1
  \quad  swap A[i] with A[t]
 \quad   {\bf return} i
}
}

From the algorithm we can see that when the pivot is the $k$-th smallest, there are $k-1$ elements smaller than the pivot. As a result, there are $k-1$ swaps in the {\tt if} statement of the algorithm {\tt partition}. Including the last swap outside the {\tt if} statement, there are $k$ total swaps. We have the recurrence relation for its probability generating function 
$$
P_n(t) = \frac{1}{n} \sum_{k=1}^n P_{k-1}(t) P_{n-k}(t) t^k
$$
with the initial condition $P_0(t)=P_1(t)=1$. 

\begin{theorem}
The expectation of the number of swaps of Quicksort for a list of length $n$ under Variant III
$$
E[X_n] = (n+1) H_1(n) - \frac{4}{3} n -\frac{1}{3}.
$$
\end{theorem}

\begin{theorem}
The variance of the number of swaps of Quicksort for a list of length $n$ under Variant III
$$
var[X_n] = 2n^2 + \frac{187}{45} n +\frac{7}{45} - \frac{2}{3n} - (n^2+2n+1) H_2(n) - (n+1) H_1(n).
$$
\end{theorem}

Note that for this variant, and some other ones in the remainder of the chapter, to find out its explicit expression of moments, we may need to modify our educated guess to a ``rational function'' of $n$ and $H_k(n)$ for some $k$ (see procedure {\tt QsMFnRat} and {\tt QsMFnRatG}). Moreover, sometimes when we solve the equations obtained by equalizing the guess with the empirical data, some initial terms should be disregarded since the increasing complexity of the definition of the Quicksort algorithms may lead to the ``singularity'' of the first several terms of moments. Usually, the higher the moment is, the more initial terms should be disregarded. 

\subsection{Variant IV}
In Variant III, every time when {\tt A[j] < pivot}, we swap {\tt A[i]} with {\tt A[j]}. However, it is a waste to swap them when $i=j$. If we modify the algorithm such that a swap is performed only when the indexes $i \neq j$, the expected cost will be reduced. Besides, if the pivot is actually the largest element, there is no need to {\tt swap A[i] with A[t]} in the partition algorithm. To popularize Maple in mathematical and scientific research, we attach Maple code for the partition part here, the {\tt ParIP} procedure in the package {\tt QuickSort.txt}, in which {\tt Swap(S, i, j)} is a subroutine to swap the $i$-th element with the $j$-th in a list $S$.
{\obeylines
{\tt 
ParIP:=proc(L) local pivot,i,j,n,S: 
n:=nops(L):
pivot:=L[n]:
S:=L:
i:=1:
for j from 1 to n-1 do
\quad  if S[j]<=pivot then
 \qquad   if i<>j then
  \quad \qquad    S:=Swap(S, i, j):
 \quad \qquad     i:=i+1:
  \qquad  else
  \quad \qquad     i:=i+1:
\qquad    fi:
\quad  fi:
od:
if i<>n then
\quad  return Swap(S, i, n), i:
else
\quad  return S, i:
fi:
end:
}
} 

\begin{lemma} Let $Y_n(k)$ be the number of swaps needed in the first partition step in an in-place Quicksort without swapping the same index for a list $L$ of length $n$ when the pivot is the $k$-th smallest element, then
$$
Y_n(k) = 
\begin{cases} 
   |\, \{i \in [n]\,|\, L[i] \leq pivot \wedge \exists j<i, L[j] > pivot  \} \, |  & k<n \\
   0 & k=n
\end{cases}.
$$
\end{lemma}

\begin{proof}
When $k=n$, it is obvious that for each search index $j$, the condition {\tt S[j] <= pivot} is satisfied, hence the insertion index $i$ is always equal to $j$, which means there is no swap inside the loop. Since eventually $i=n$, there is also no need to swap the pivot with itself. So the number of swaps is $0$ in this case. 

When $k<n$, notice that the first time $i$ is smaller than $j$ is when we find the first element greater than the pivot. After that, $i$ will always be less than $j$, which implies that for each element smaller than the pivot and the pivot itself, one swap is performed. 
\end{proof}

The Maple procedure {\tt IPProb(n,k,t)} takes inputs $n,k$ as defined above and a symbol $t$, and outputs the probability generating function $Q(n,k,t)$ for the number of swaps in the first partition when the length of the list is $n$ and the last element, which is chosen as the pivot, is the $k$-th smallest.

When $k<n$, the probability that there are $s$ swaps is
$$
\binom{k-1}{k-s} \frac{(k-s)! (n-k) (n-k-2+s)! }{ (n-1)!} = \frac{n-k}{n-1} \frac{\binom{k-1}{k-s}}{\binom{n-2}{k-s}}.
$$
Therefore the probability generating function 
$$
Q(n,k,t) = \sum_{s=1}^k \frac{n-k}{n-1} \frac{\binom{k-1}{k-s}}{\binom{n-2}{k-s}} t^s.
$$
The recurrence relation for the probability generating function $P_n(t)$ of the total number of swaps follows immediately:
$$
P_n(t) = \frac{1}{n} \sum_{k=1}^n P_{k-1}(t) P_{n-k}(t) Q(n,k,t)
$$
with the initial condition $P_0(t)=P_1(t)=1$.

\begin{theorem}
The expectation of the number of swaps of Quicksort for a list of length $n$ under Variant IV
$$
E[X_n] = (n+2) H_1(n) - \frac{5}{2} n -\frac{1}{2}.
$$
\end{theorem} 

\begin{theorem}
 The variance of the number of swaps of Quicksort for a list of length $n$ under Variant IV
$$
var[X_n] = 2n^{2}-{\frac {215}{12}}n+\frac{1}{12}+(11n+14)H_1(n) -(n^2-2n-2)H_2(n) - (2n+2) H_1(n)^2.
$$
\end{theorem} 

Comparing these results with {\bf Theorem 4.5} and {\bf Theorem 4.9}, it shows that Variant IV has better average performances, notwithstanding the ``broader'' definition of ``swap'' in the first two subsections. And of course, it is better than the in-place Quicksort which swaps the indexes even when they are the same. We fully believe that the additional cost to check whether the insertion and search indexes are the same is worthwhile.

\subsection{Variant V}
This variant might not be practical, but we find that it is interesting as a combinatorial model. As is well-known, if a close-to-median element is chosen as a pivot, the Quicksort algorithm will have better performance than average in this case. Hence if additional information is available so that the probability distribution of chosen pivots is no longer a uniform distribution but something Gaussian-like, it is to our advantage. 

Assume that the list is a permutation of $[n]$ and we are trying to sort it, pretending that we do not know the sorted list must be $[1,2, \dots, n]$. Now the rule is that we choose the first and last number in the list, look at the numbers and choose the one which is closer to the median. If the two numbers are equally close to the median, then choose one at random.

Without loss of generality, we can always assume that the last element in the list is chosen as the pivot; otherwise we just need to run the algorithm in the last subsection ``in reverse'', putting both the insertion and search indexes on the last element and letting larger elements be swapped to the left side of the pivot, etc. So the only difference with Variant IV is the probability distribution of $k$, which is no longer $\frac{1}{n}$ for each $k \in [n].$

Considering symmetry, $Pr^{(n)}(pivot = k) = Pr^{(n)}(pivot = n+1-k)$, so we only need to consider $1 \leq k \leq (n+1)/2$. When $n$ is even, let $n=2m$. Then $Pr^{(n)}(pivot = k) = \frac{4k-3}{(2m-1)2m}$ . When $n$ is odd, let $n = 2m-1$, then $Pr^{(n)}(pivot = k) = \frac{4k-3}{(2m-1)(2m-2)}$ when $k<m$ and $Pr^{(n)}(pivot = m) = \frac{2}{2m-1}$.

With this minor modification, the recurrence relation for $P_n(t)$ follows.
$$
P_n(t) =  \sum_{k=1}^n P_{k-1}(t) P_{n-k}(t) Q(n,k,t) Pr^{(n)}(pivot = k)
$$
with the initial condition $P_0(t)=P_1(t)=1$.

Though an explicit expression seems difficult in this case, we can still analyze the performance of the algorithm by evaluating its expected number of swaps. By exploiting the Maple procedure {\tt MomFn(f,t,m,N)}, which inputs a function name $f$, a symbol $t$, the order of the moment $m$ and the upper bound of the length of the list $N$ and outputs a list of $m$-th moments for the Quicksort described by the probability generating function $f$ of lists of length $1,2, \dots, N$, we find that Variant V has better average performance than Variant IV when $n$ is large enough. For example, {\tt MomFn(PQIP, t, 1, 20)} returns
$$
[0,\frac{1}{2},\frac{7}{6},2,{\frac {179}{60}},{\frac {41}{10}},{\frac {747}{140}},{
\frac {187}{28}},{\frac {20459}{2520}},{\frac {1013}{105}},{\frac {
312083}{27720}},{\frac {25631}{1980}},{\frac {353201}{24024}},{\frac {
1488737}{90090}},
$$
$$
{\frac {6634189}{360360}},{\frac {814939}{40040}},{
\frac {273855917}{12252240}},{\frac {4983019}{204204}},{\frac {
97930039}{3695120}},{\frac {20210819}{705432}}],
$$
and {\tt MomFn(PQIPk, t, 1, 20)} returns 
$$
[0,\frac{1}{2},\frac{4}{3},{\frac {20}{9}},{\frac {155}{48}},{\frac {1957}{450}},{
\frac {2341}{420}},{\frac {4055}{588}},{\frac {55829}{6720}},{\frac {
794}{81}},{\frac {630547}{55440}},{\frac {170095}{13068}},{\frac {
12735487}{864864}},
$$
$$
{\frac {3864281}{234234}},{\frac {2521865}{137592}}
,{\frac {36424327}{1801800}},{\frac {4343228489}{196035840}},{\frac {
107768347}{4463316}},{\frac {15673532207}{598609440}},{\frac {
1136599735}{40209624}}].
$$
We notice that for $n \geq 14$, Variant V consistently has better average performance. From this result we can conclude that it is worth choosing a pivot from two candidates since the gains of efficiency are far more than its cost. 

Moreover, we can obtain a recurrence for the expected number $x_n$ of the random variable $X_n$. The {\tt Findrec(f,n,N,MaxC)} procedure in the Maple package {\tt Findrec.txt} inputs a list of data $L$, two symbols $n$ and $N$, where $n$ is the discrete variable, and $N$ is the shift operator, and $MaxC$ which is the maximum degree plus order. {\tt Findrec(MomFn(PQIPk, t, 1, 80), n, N, 11)} returns

\begin{figure}[h!]
  \includegraphics[width=\textwidth]{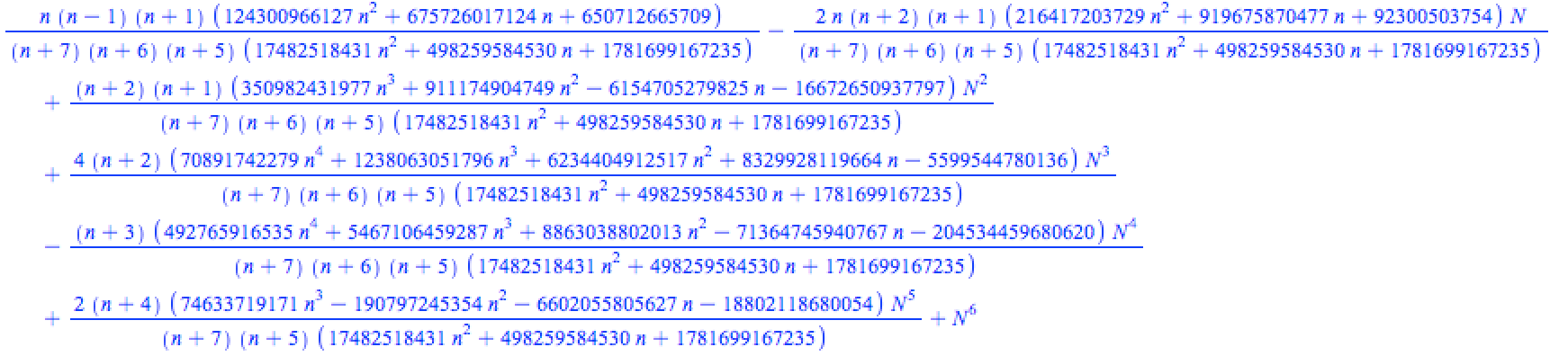}
  \caption{The Recurrence Relation for the Expected Number of Swaps}
\end{figure}

As previously mentioned, $N$ is the shift operator. Since this formula is too long, to see its detailed interpretation, please look at {\bf Theorem 4.23} as reference. 

We can also look at more elements to choose the middle-most one as the pivot. In case that we do not want to store so much information, 
some other variants involving a ``look twice'' idea could be that if the first selected element is within some satisfactory interval, e.g., $[\frac{n}{4}, \frac{3n}{4}]$ for a permutation of $n$, then it is chosen
as the pivot. Otherwise  we choose a second element as the pivot without storing information about the first one. It is also likely to improve the algorithm with ``multiple looks'' to choose the pivot and the requirement to choose the current element as the pivot without continuing to look at the next one could vary and ideally the requirement should be more relaxed as the number of ``looks'' increases. We also refer the readers to \cite{25} where P. Kirschenhofer, H. Prodinger and C. Martinez chose the median of a random sample of three elements as the pivot and obtained explicit expressions for this variant's performance with methods of hypergeometric differential equations. In general, it is ideal to have a probability distribution of pivots where an element closer to the median is more likely to be chosen.

\section{Explorations for Multi-Pivot Quicksort}
With the same general idea as the 1-pivot Quicksort, it is natural to think about ``what if we choose more pivots." In $k$-pivot Quicksort, $k$ indexes are chosen and the correspondent elements become pivots. The $k$ pivots need to be sorted and then the other elements will be partitioned into one of the $k+1$ sublists. Compared to 1-pivot Quicksort, multi-pivot Quicksort is much more complicated because we need to determine how to sort the pivots, how to allocate each non-pivot element to the sublist they belong to and how to sort a sublist containing less than $k$ elements. Therefore, there are a few multi-pivot Quicksort variants. We refer the reader to \cite{20} for other versions of multi-pivot. When $k$ is large, it seems difficult to have an in-place version, so we mainly consider the random variable, the number of comparisons $C_n$, in this section since a swap might be difficult to define in this case. 

\subsection{Dual-Pivot Quicksort}
Let's start from the simplest multi-pivot Quicksort: dual-pivot. The model for dual-pivot Quicksort is that two pivots $p_1$ and $p_2$ are randomly chosen. After one comparison, they are sorted, say $p_1<p_2$. Non-pivot elements should be partitioned into one of the three sublists $L_1, L_2$ and $L_3$. $L_1$ contains elements smaller than $p_1$, $L_2$ contains elements between $p_1$ and $p_2$ while $L_3$ contains elements greater than $p_2$. Each non-pivot element is compared with $p_1$ first. If it is smaller than $p_1$, we are done. Otherwise it is compared with $p_2$ to determine whether it should go to $L_2$ or $L_3$.

Given that the list contains $n$ distinct elements and the two pivots are the $i$-th and $j$-th smallest element $(i<j)$, we need one comparison to sort the pivot and $i-1+2(n-i-1) = 2n-i-3$ comparisons to distribute non-pivot elements to the sublists. Hence the total number of comparison is $2n-i-2$ and the recurrence relation for the probability generating function $P_n(t)$ of the total number of comparisons $C_n$ of dual-pivot Quicksort is
$$
P_n(t) = \frac{1}{\binom{n}{2}} \sum_{j=2}^n \sum_{i=1}^{j-1}  P_{i-1}(t) P_{j-i-1}(t) P_{n-j}(t) t^{2n-i-2}
$$
with the initial condition $P_0(t)=P_1(t)=1$ and $P_2(t) = t$.

The above recurrence is implemented by the procedure {\tt PQc2}. With the aforementioned Maple procedure {\tt QsMFn} it is easy to get the following theorems.

\begin{theorem}
The expectation of the number of comparisons in dual-pivot Quicksort algorithms is
$$
E[C_n]=2(n+1) H_1(n) - 4n .
$$
\end{theorem} 

\begin{theorem}
The variance of the number of comparisons in dual-pivot Quicksort algorithms is
$$
var[C_n] = n ( 7\,n+13 )  \, - \, 2\,(n+1)\, H_{{1}} ( n )  -4\, ( n+1 ) ^{2}H_{{2}} ( n )  .
$$
\end{theorem} 

\begin{theorem}
The third moment about the mean of $C_n$ is 
$$
-n ( 19\,{n}^{2}+81\,n+104 ) +H_{{1}} ( n ) 
 ( 14\,n+14 ) +12\, ( n+1 ) ^{2}H_{{2}} ( n
 ) +16\, ( n+1 ) ^{3}H_{{3}} ( n )  .
$$
\end{theorem} 

\begin{theorem} 
The fourth moment about the mean of $C_n$ is 
$$
\frac{1}{9}\,n ( 2260\,{n}^{3}+9658\,{n}^{2}+15497\,n+11357 ) -2\,
 ( n+1 )  ( 42\,{n}^{2}+78\,n+77 ) H_{{1}}
 ( n ) 
$$
$$
+12\, ( n+1 ) ^{2} ( H_{{1}} ( n
 )  ) ^{2}+ ( -4\, ( 42\,{n}^{2}+78\,n+31
 )  ( n+1 ) ^{2}+48\, ( n+1 ) ^{3}H_{{1}}
 ( n )  ) H_{{2}} ( n ) 
$$
$$
+48\, ( n+1
 ) ^{4} ( H_{{2}} ( n )  ) ^{2}-96\,
 ( n+1 ) ^{3}H_{{3}} ( n ) -96\, ( n+1
 ) ^{4}H_{{4}} ( n ) .
$$
\end{theorem} 

As usual, any higher moment can be easily obtained with a powerful silicon servant. Careful readers may notice that the above four theorems are exactly the same as the ones in Section 4.2. It is natural to ask whether they indeed have the same probability distribution. The answer is yes. In Section 4.1.2 of \cite{19} the author gives a sketch of proof showing that 1-pivot and dual-pivot Quicksorts's numbers of comparisons satisfy the same recurrence relation and initial condition. From an experimental mathematical point of view, a semi-rigorous proof obtained by checking sufficiently many special cases is good enough for us. For example, the first 10 probability generating function, $P_n(t)$ for $1 \leq n \leq 10$ can be calculated in a nanosecond by entering {\tt [seq(PQc2(i, t), i = 1..10)]} and we have
$$
P_1(t) = 1,
$$
$$
P_2(t) = t,
$$
$$
P_3(t) = \frac{2}{3} t^3 + \frac{1}{3} t^2,
$$
$$
P_4(t) = \frac{1}{3}{t}^{6}+\frac{1}{6}{t}^{5}+\frac{1}{2}{t}^{4},
$$
$$
P_5(t) = \frac{2}{15}{t}^{10}+\frac{1}{15}{t}^{9}+\frac{1}{5}{t}^{8}+\frac{4}{15} t^7 +\frac{1}{3} t^6,
$$
$$
\vdots
$$
which are exactly the same as the probability generating function for the number of comparisons of 1-pivot Quicksort. 

In conclusion, in terms of probability distribution of the number of comparisons, the dual-pivot Quicksort does not appear to be better than the ordinary 1-pivot Quicksort. 

As for the random variable $X_n$, the number of swaps, it depends on the specific implementation of the algorithm and the definition of a ``swap." As a toy model, we do an analogue of Section 4.3.1. Assume the list is a permutation of $[n]$. The first and last elements are chosen as the pivot. Let's say they are $i$ and $j$. If $i>j$ then we swap them and still call the smaller pivot $i$. For each element less than $i$, we move it to the left of $i$, and for each element greater than $j$, we move it to the right of $j$ and call this kind of operations a swap. 

The recurrence relation for the probability generating function of $X_n$ is
$$
P_n(t) = \frac{1}{\binom{n}{2}} (\frac{1}{2} + \frac{1}{2} t) \sum_{j=2}^n \sum_{i=1}^{j-1} P_{i-1}(t) P_{j-i-1}(t) P_{n-j}(t) t^{n-1+i-j}
$$
with the initial conditions $P_0(t) = P_1(t) =1$ and $P_2(t) =\frac{1}{2} + \frac{1}{2} t. $

\begin{theorem}
The expectation of the number of swaps in the above dual-pivot Quicksort variant is
$$
E[X_n]=\frac{4}{5} (n+1) H_1(n) - \frac{39}{25} n -\frac{1}{100}.
$$
\end{theorem} 
Note that this result is better than those in Sections 4.3.1 and 4.3.2.

\subsection{Three-Pivot Quicksort}
As mentioned at the beginning of this section, to define a 3-pivot Quicksort, we need to define 1) how to sort the pivots, 2) how to partition the list and 3) how to sort a list or sublist containing less than 3 pivots. Since this chapter is to study Quicksort, we choose 1-pivot Quicksort for 1) and 3). For 2), it seems that the binary search is a good option since for each non-pivot element exactly 2 comparisons with the pivots are needed. 

The Maple procedure {\tt PQs(n,t)} outputs the probability generating function for 1-pivot Quicksort of a list of length $n$. Hence during the process of sorting the pivots and partitioning the list, the probability generating function of the number of comparisons is $PQs(3,t) t^{2n-6}$, which equals to
$$
(\frac{2}{3} t^3 + \frac{1}{3} t^2) t^{2n-6} = \frac{2}{3} t^{2n-3} + \frac{1}{3} t^{2n-4}.
$$
So the recurrence relation for the probability generating function $P_n(t)$ of the total number of comparisons for 3-pivot Quicksort of a list of length $n$ is
$$
P_n(t) = \frac{1}{\binom{n}{3}} \sum_{k=3}^n \sum_{j=2}^{k-1} \sum_{i=1}^{j-1} P_{i-1}(t) P_{j-i-1}(t) P_{k-j-1}(t) P_{n-k}(t) (\frac{2}{3} t^{2n-3} + \frac{1}{3} t^{2n-4} )
$$
with initial conditions $P_0(t) = P_1(t) =1, P_2(t) = t$ and $P_3(t) =\frac{2}{3} t^3 + \frac{1}{3} t^2 $. This recurrence is implemented by the procedure {\tt PQd3}.

The explicit expression seems to be difficult to obtain in this case, but numerical tests imply that 3-pivot Quicksort has better performances than dual-pivot, and of course 1-pivot since it is indeed equivalent to dual-pivot.  

By exploiting the Maple procedure {\tt MomFn(f,t,m,N)} again, we can compare the expectations of different Quicksort variants. 

For instance, {\tt MomFn(PQc2, t, 1, 20)} returns 
$$
[0,1,\frac{8}{3},{\frac {29}{6}},{\frac {37}{5}},{\frac {103}{10}},{\frac {472}{35}},{\frac {2369}{140}},{\frac {2593}{126}},{\frac {30791}{1260}},{
\frac {32891}{1155}},{\frac {452993}{13860}},{\frac {476753}{12870}},{
\frac {499061}{12012}},
$$
$$
{\frac {2080328}{45045}},{\frac {18358463}{
360360}},{\frac {18999103}{340340}},{\frac {124184839}{2042040}},{
\frac {127860511}{1939938}},{\frac {26274175}{369512}}],
$$
and {\tt MomFn(PQd3, t, 1, 20)} returns
$$
[0,1,\frac{8}{3},\frac{14}{3},{\frac {106}{15}},{\frac {49}{5}},{\frac {64}{5}},{\frac {561}{35}},{\frac {1226}{63}},{\frac {5192}{225}},{\frac {465316}{17325}},{\frac {533509}{17325}},{\frac {714008}{20475}},{\frac {
61615768}{1576575}},
$$
$$
{\frac {342234824}{7882875}},{\frac {754600981}{
15765750}},{\frac {1404956027}{26801775}},{\frac {15298397599}{
268017750}},{\frac {31489234438}{509233725}},{\frac {1697926310039}{
25461686250}}].
$$
We notice that for each fixed $n>3$, 3-pivot Quicksort's average performance is better than 2-pivot and 1-pivot. This numerical test is also possible for all previous Quicksort variants but seems unnecessary when the explicit expressions are easily accessible. 

As in Section 4.3.5, a recurrence relation of the expected number of comparisons can be obtained. {\tt Findrec(MomFn(PQd3, t, 1, 40),n,N,8)} returns 
$$
{\frac { \left( 3\,n+4 \right)  \left( {n}^{2}-5\,n+12 \right) }{
 \left( n+4 \right)  \left( n+3 \right)  \left( 3\,n+1 \right) }}-{
\frac { \left( 12\,{n}^{4}+13\,{n}^{3}-12\,{n}^{2}+59\,n+24 \right) N
}{ \left( 3\,n+1 \right)  \left( n+4 \right)  \left( n+3 \right) 
 \left( n+2 \right) }}
$$
$$
 +3\,{\frac { \left( n+1 \right)  \left( 6\,n+5
 \right) n{N}^{2}}{ \left( n+4 \right)  \left( n+3 \right)  \left( 3\,
n+1 \right) }}-{\frac { \left( n+1 \right)  \left( 12\,n+7 \right) {N}
^{3}}{ \left( n+4 \right)  \left( 3\,n+1 \right) }}+{N}^{4},
$$
which leads to the following theorem.

\begin{theorem} 
The expected number of comparisons $C_n$ of 3-pivot Quicksort for a list of length $n$ satisfies the following recurrence relation:
$$
C_{n+4} = {\frac { \left( n+1 \right)  \left( 12\,n+7 \right) }{ \left( n+4 \right)  \left( 3\,n+1 \right) }} C_{n+3} - 3\,{\frac { \left( n+1 \right)  \left( 6\,n+5 \right) n}{ \left( n+4 \right)  \left( n+3 \right)  \left( 3\,n+1 \right) }} C_{n+2}
$$
$$
+{\frac { \left( 12\,{n}^{4}+13\,{n}^{3}-12\,{n}^{2}+59\,n+24 \right) 
}{ \left( 3\,n+1 \right)  \left( n+4 \right)  \left( n+3 \right) 
 \left( n+2 \right) }} C_{n+1} - {\frac { \left( 3\,n+4 \right)  \left( {n}^{2}-5\,n+12 \right) }{
 \left( n+4 \right)  \left( n+3 \right)  \left( 3\,n+1 \right) }} C_{n}.
$$
\end{theorem}
The recurrence relations for higher moments are also obtainable, but a long enough list of data is needed. 

\subsection{$k$-Pivot Quicksort}
More generally, $k$-pivot Quicksort can be considered with the convention that 1) the $k$ pivots are sorted with 1-pivot Quicksort, 2) binary search is used to partition the list into $k+1$ sublists, 3) we use 1-pivot Quicksort to sort lists with length less than $k$.

In the package {\tt QuickSort.txt} the procedure {\tt PQck(n, t, k)} outputs the probability generating function for the number of comparisons of $k$-pivot Quicksort where each element is compared with pivots in a linearly increasing order. Obviously this is not efficient when $k$ is large. However, the problem for binary search is that when $k \neq 2^i-1$ for some $i$, it is hard to get an expression for the number of comparisons in the binary search step since the number highly depends on the specific implementation where some boundary cases may vary and the floor and ceiling functions will be involved, which leads to an increasing difficulty to find the explicit expressions for moments. 

There is a procedure {\tt QsBC(L, k)} which inputs a list of distinct numbers $L$ and an integer $k$ representing the number of pivots and outputs the sorted list and the total number of comparisons. For convenience of Monte Carlo experiments, we use {\tt MCQsBC(n, k, T)} where $n$ is the length of the list, $k$ is the number of pivots and $T$ is the number of times we repeat the experiments. Because of the limit of computing resources, we only test for $k=3,4,5,6$ and $n=10, 20 ,30 ,40 ,50.$
$$
{\tt [seq(MCQsBC(10*i, 3, 100), i = 1 .. 5)]} = [22.95, 65.75, 118.71, 178.28, 239.45],
$$

$$
{\tt [seq(MCQsBC(10*i, 4, 100), i = 1 .. 5)]} = [23.78, 67.77, 120.91, 180.35, 251.19],
$$

$$
{\tt [seq(MCQsBC(10*i, 5, 100), i = 1 .. 5)]} = [23.54, 65.74, 119.59, 178.36, 241.03],
$$

$$
{\tt [seq(MCQsBC(10*i, 6, 100), i = 1 .. 5)]} = [23.14, 66.22, 120.07, 176.43, 236.46],
$$

Our observation is that for large enough $n$, the more pivots we use, the less comparisons are needed. However, when $k$ is too close to $n$, the increase of pivots may lead to inefficiency. 

\section{Limiting Distribution}
The main purpose of this chapter is to find explicit expressions for the moments of the number of swaps or comparisons of some variants of Quicksort, to compare their performances and to explore more efficient Quicksort algorithms. However, it is also of interest to find more moments for large $n$ and calculate their floating number approximation of the scaled limiting distribution. 

As mentioned in \cite{12}, if we are only interested in the first few moments, then it is wasteful to compute the full probability generating function $P_n(t).$
Let $t=1+w$ and use the fact that
$$
P_n(1+w)  = \sum_{r=0}^{\infty} \frac{f_r(n)}{r!} w^r
$$
where $f_r(n)$ are the factorial moments. The straight moments $E[X_n^r]$, and the moments-about-the-mean,
$m_r(n)$ follow immediately. 

As a case study, let's use Variant IV from Section 4.3.4 as an example. The recurrence relation is 
$$
P_n(1+w)  = \frac{1}{n} \sum_{k=1}^{n} P_{k-1}(1+w) P_{n-k}(1+w) Q(n,k,1+w), 
$$
where $Q$ is as defined in Section 4.3.4.

Since only the first several factorial moments are considered, in each step truncation is performed and only the first several coefficients in $w$ is kept. With this method we can get more moments in a fixed time. The procedure {\tt TrunIP} implements the truncated factorial generating function. 

With the closed-form expressions for both the expectation, $c_n$, and the variance $m_2(n):=var(X_n)$, the scaled random variable $Z_n$ is defined as follows.
$$
Z_n := \frac{X_n -c_n}{\sqrt{m_2(n)}} .
$$

We are interested in the floating point approximations of the limiting distribution $\lim_{n \rightarrow \infty} Z_n$. Of course its expectation is $0$ and its variance is $1$.

For instance, if we'd like to know the moments up to order 10, {\tt TrunIP(100, z, 10)} returns
$$
1+{\frac {7617634712836831344646726224164628686543}{
27341323619495089084130905464828354336}} z
$$
$$
+{\frac {
1169146867836246319480317311960440606057785761234433183813484643}{29517287662514914280390084303910684938635848245569645536000}} z^2
$$
$$
+ \dots \quad.
$$
For $1 \leq r \leq 10$, the coefficient of $z^r$ times $r!$ is the $r$-th factorial moment. By 
$$
E[X^r] = \sum_{j=0}^r {r\brace j} E[(X)_j]
$$
where the curly braces denote Stirling numbers of the second kind, we can get the raw moments. And with the procedure {\tt MtoA} a sequence of raw moments are transformed to moments about the mean. Divided by $m_2(n)^{\frac{r}{2}}$, the 3rd through 10th moments
in floating point approximations are
$$
[0.7810052982, 3.942047050, 9.146681877, 37.12169647, 137.7143092,
$$
$$
613.5286860, 2872.409923, 14709.75560].
$$
The same technique can be applied to other variants of Quicksort in this chapter and we leave this to interested readers.

\section{Remarks}
In such a rich and active research area as Quicksort, there are still several things we could think about to improve the algorithms' performances. Just to name a few, in 2-pivot Quicksort when we compare non-pivot elements with the pivots to determine which sublist they belong to, if the history is tracked, we might be able to use the history to determine which pivot to compare with first for the next element. The optimal strategy would vary with the additional information about the range of the numbers or the relative ranking of the two pivots among all the elements. 

As for $k$-pivot Quicksort, our naive approach only distinguishes two cases: whether the currently to-be-sorted list has length less than $k$ or not. If the length is less than $k$, we use 1-pivot Quicksort; otherwise we still choose $k$ pivots. However, we might be able to improve the performance if the number of pivots varies according to the length of the to-be-sorted list or sublist. Let's say there is a function $g(n)$, where $n$ is the length of the list. So we pick $g(n)$ pivots at the beginning. After we obtain the $g(n)+1$ sublists with length $n_i, 1 \leq i \leq g(n)+1$, for each one of them, we choose $g(n_i)$ pivots. It would be interesting whether we can find an optimal $g$ in terms of its average performance. Additionally, when $k$ is large, it might make sense to use multi-pivot Quicksort to sort the $k$ pivots as well. 

Of course, it is also interesting to study the explicit expressions of the numbers of swaps in multi-pivot Quicksort. But it appears to be dependent on the specific implementation of the algorithm so it is of significance to look for variants which save time and space complexity. 

The main results of this chapter are those explicit expressions of moments and recurrence relations for either the number of comparisons or the number of swaps of various Quicksort variants. Though all of their asymptotics are $O(n \log n)$, the constant before this term varies a lot and some comparisons of these variants are also discussed. When there is difficulty getting the explicit expressions, numerical tests and Monte Carlo experiments are performed. We also have a demonstration on how to get more moments and find the numeric approximation of the scaled limiting distribution. 

Nevertheless, more important than those results is the illustration of a methodology of experimental mathematics. From ansatzes and sufficient data we have an alternative way to obtain results that otherwise might be extremely difficult or even impossible to get via traditional human approaches to algorithm analysis.

\chapter{Peaceable Queens Problem}
This chapter is adapted from \cite{50}, which has been published on {\it Experimental Mathematics}. It is also available on arXiv.org, number 1902.05886.

\section{Introduction}

One of the fascinating problems described in the recent article \cite{39}, about the great  {\bf On-Line Encyclopedia of Integer Sequences}, and in the beautiful and insightful video \cite{40}
is the {\it peaceable queens problem}.
It was chosen, by popular vote, to be assigned the {\bf milestone} `quarter-million' A-number, {\tt A250000}.

The question is the following:

{\it What is the maximal number, $m$, such that it is possible to place $m$ white queens and $m$ black queens on an
$n \times n$ chess board, so that no queen attacks a queen of the opposite color.}

Currently only fifteen terms are known:
\begin{align*}
n: && 1 && 2 && 3 && 4 && 5 && 6 && 7 && 8 && 9 && 10 && 11 && 12 && 13 && 14 && 15\\
a(n): && 0 && 0 && 1 && 2 && 4 && 5 && 7 && 9 && 12 && 14 && 17 && 21 && 24 && 28 && 32
\end{align*}

In this chapter, we'd like to consider this peaceable queens problem as a continuous problem by normalizing the chess board to be the unit square $U:=[0,1]^2 = \{(x,y) \, | \, 0 \leq x, y \leq 1 \}$. Let $W \subseteq U$ be the region where white queens are located. Then the non-attacking region $B$ of $W$ can be defined as
$$
B = \{(x,y) \in U \,| \, \forall (u,v) \in W, x \neq u,  y \neq v, x+y \neq u+v,  y-x \neq v-u \}.
$$
So the continuous version of the peaceable queens problem is to find 
$$
\max_{W \in 2^U} ( \min(\textrm{Area} (W), \textrm{Area}(B)) ).
$$

Considering that the queen is able to move any number of squares vertically, horizontally and diagonally, it is reasonable to let $W$ be a convex polygon or a disjoint union of convex polygons whose boundary consists of vertical, horizontal and slope $\pm  1$ line segments, otherwise we can increase the area of white queens without decreasing the area of black queens. 

In this chapter, we use a list $L$ of lists $[\,[a_1, b_1]\,,\, [a_2, b_2]\,,\, \dots\,,\, [a_n, b_n]\,]$ to denote the $n$-gon whose vertices are the $n$ pairs in the list $L$ and whose sides are the straight line segments connecting $[a_i, b_i]$ and $[a_{i+1}, b_{i+1}], (1 \leq i \leq n-1) $, and $[a_n, b_n]$ and $[a_1, b_1]$.

This chapter is organized as follows. At first we look at Jubin's construction and prove that it is a local optimum. Though there is no rigorous proof, we conjecture and reasonably believe that it is indeed a global optimum at least for ``the continuous chess board", after numerous experiments with one, two and more components. Then we consider the optimal case under more restrictions, or under certain configurations, e.g., only one component or two identical squares or two identical triangles, etc. In some cases, the exact optimal parameters and areas can be obtained. Note that in this chapter's figures, for convenience of demonstration, the color red is used to represent white queens and blue is for black queens.

\section{Jubin's Construction}

As mentioned in \cite{39} (and \cite{38}, sequence A250000), it is conjectured that Benoit Jubin's construction 
given in Fig. 5 of \cite{39}, see also here:

\noindent {\tt http://sites.math.rutgers.edu/\~{}zeilberg/tokhniot/peaceable/P1.html}

\noindent or Figure 5.1, 
is  optimal for $n \geq 10$. Its value is $\lfloor \frac{7n^2}{48} \rfloor$.

\begin{figure}[h!]
  \center
  \includegraphics[width=0.75\textwidth]{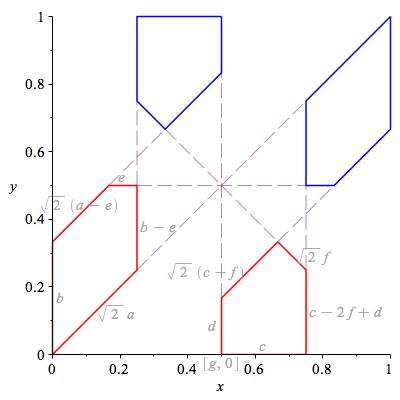}
  \caption{Benoit Jubin's Construction for a Unit Square}
\end{figure}

While we are, at present, unable to prove this,
we did manage to prove that when one generalizes Jubin's construction and replaces the
sides of the two pentagons with arbitrary parameters (of course subject to the obvious constraints so that both white and black queens reside in two pentagons),
then Jubin's construction is indeed (asymptotically) optimal, i.e. in the limit as $n$ goes to infinity.

\begin{lemma}
Normalizing the chess board to be  the unit square $\{(x,y) \, | \, 0 \leq x, y \leq 1 \}$,
if the white queens are placed in the union of the interiors of the following two pentagons
$$
[\, [0,0] \, , \, [a,a] \, , \, [a,a+b-e] \, , \, [a-e,a+b-e] \, , \, [0,b] \,   ],
$$
and
$$
[\, [g,0] \, , \, [g+c,0] \, , \, [g+c,c-2\,f+d] \, , \, [g+c-f,c-f+d] \, , \, [g,d] \,  ],
$$
where $a,b,c,d,e,f,g$ are between $0$ and $1$ and all coordinates and side lengths in Fig. 5.1 are non-negative and appropriate so that black queens also reside in two pentagons, then the black queens are located in the
interiors of the pentagons
$$
[\, [g,1] \, , \, [a,1] \, , \, [a,g+2\,c-2\,f+d-a] \, , \, 
[\frac{1}{2} \,g+ \frac{1}{2}\,d-\frac{1}{2}\,b+c-f, \frac{1}{2}\,g+ \frac{1}{2}\,d+ \frac{1}{2}\,b+c-f] \, , 
\,
$$
$$
[g,g+b] \, ] ,
$$
and
$$
[\, [1,1] \, , \, [g+c,g+c] \, , \, [g+c,a+b-e] \, , \, [a+b-e+g-d,a+b-e] \, , \, [1,1+d-g] \,  ]  .
$$
\end{lemma}

\begin{proof}
Since we only consider cases when the black queens also reside in two pentagons, this requirement provides natural constraints for these parameters $a,b,c,d,e,f,g$. Just to name a few, $a \leq g$ because the two pentagons are not overlapped, $g+c \leq 1$ because the right pentagon of white queens should entirely reside in the unit square, $d\leq g$ because otherwise the right pentagon of black queens will not exist and $c-f+d$, which is the $y$-coordinate of the highest point in the right pentagon of white queens, cannot be too large to ensure the right pentagon of black queens does not degenerate to a parallelogram. In these constraints we always use ``$\leq$" instead of ``$<$" so that the Lagrange multipliers will be able to work in a closed domain. 

With school geometry, it is obvious that black queens cannot reside in $0 \leq x < a$ since it is attacked by the left pentagon of white queens. Similar arguments work for the area $0<y \leq a+b$ and $g<x<g+c$. Now the leftover on the unit square is a union of two rectangles. By excluding $x+y < g+2c-2f+d, 0<y-x<b$ and $y-x<d-g$, these two rectangles are shaped into two pentagons and the coordinates of their vertices follow immediately.
\end{proof}

\begin{lemma} The area of the white queens is
$$
ab \, - \,  \frac{1}{2}\,{e}^{2}+cd \, + \,  \frac{1}{2}\,{c}^{2}-{f}^{2} ,
$$
while the area of the black queens is
$$
-a-\frac{3}{4}\,{d}^{2}+2\,g-d-cd-ab-{f}^{2}- \frac{1}{2}\,{e}^{2}-\frac{3}{2}\,{c}^{2}+2\,bc-2\,af+3\,ac+2\,ad+2\,cf-ec-ed+be
$$
$$
+ae-bf+fd+ \frac{3}{2}\,bd-{a}^{2} - \frac{3}{4}\,{b}^{2}-2\,g c+ \frac{1}{2}\,g d- \frac{1}{2}\,g b+ag+gf- \frac{7}{4} \,{g}^{2} .
$$
\end{lemma}

\begin{proof}
For white queens, the left pentagon is a rectangle minus two triangles. Hence the area is 
$$
a(a+b-e) - \frac{1}{2} a^2 - \frac{1}{2} (a-e)^2 = ab - \frac{1}{2} e^2.
$$
The area of the right pentagon is 
$$
c(d+c-f) - \frac{1}{2} f^2 - \frac{1}{2} (c-f)^2 = \frac{1}{2} c^2 + cd - f^2.
$$
So the area of the white queens follows. 
For black queens, similarly, with the coordinates of the vertices in Lemma 5.1, simple calculation leads to the formula of its area. 
\end{proof}

\begin{theorem} 
The optimal case of the two-pentagon configuration is Jubin's construction.
\end{theorem}

\begin{proof}

The procedure {\tt MaxC(L,v)} in the Maple package {\tt PeaceableQueens.txt} takes a list of length 2, $L$, consisting of polynomials in the list of variables $v$, and $v$ as inputs and outputs all the extreme points of $L[1]$, subject to the constraint $L[1]=L[2]$, using Lagrange multipliers.

Optimally, the areas of the white queens and black queens should be the same. Maximizing this quantity with the procedure {\tt MaxC} under this constraint shows that the maximum value is
$$
\frac{7}{48}  ,
$$
and this is indeed achieved by Jubin's construction, in which the white queens are located inside the pentagons
$$
[\, [0,0] \, , \, [\frac{1}{4}, \frac{1}{4}] \, , \,  [ \frac{1}{4}, \frac{1}{2}]
\, , \, [\frac{1}{6}, \frac{1}{2}] \, , \, [0,\frac{1}{3}] \,  ] ,
$$
and
$$
[\, [ \frac{1}{2},0] \, , \, [\frac{3}{4},0] \, , \, [\frac{3}{4},\frac{1}{4}] \, , \,
[\frac{2}{3},\frac{1}{3}] \, , \, [ \frac{1}{2},\frac{1}{6}] \, ] ,
$$
and the black queens reside inside the pentagons
$$
[\, [ \frac{1}{2},1] \, , \, [\frac{1}{4},1]
\, , \, [\frac{1}{4},\frac{3}{4}] \, , \, [\frac{1}{3},\frac{2}{3}] \, , \,
[ \frac{1}{2},\frac{5}{6}] \,  ] ,
$$
and
$$
[\, [1,1] \, ,\ \, [\frac{3}{4},\frac{3}{4}] \, , \, [\frac{3}{4}, \frac{1}{2}] \, , \,
[\frac{5}{6}, \frac{1}{2}] \, , \, [1,\frac{2}{3}] \,  ] . 
$$
\end{proof}

It seems natural that two components are optimal because if there is only one connected component for white queens, black queens still have two connected components. From the view of symmetry, it seems good to add the other component for white queens. In the rest of the chapter, it is shown that with only one connected component it is unlikely to surpass the $\frac{7}{48}$ result. And by experimenting with three or more connected components for the white queens, it seems that it is not possible to improve on
Jubin's construction, hence we believe that it is indeed optimal (at least asymptotically). By the way, Donald Knuth kindly informed us that
what we (and the OEIS) call Jubin's construction already appears in  Stephen Ainley's delightful book {\it Mathematical Puzzles} \cite{1}, p. 31, Fig, 28(A) .

\section{Single Connected Component}
In this section, we try to find the optimal case when $W$ is a single connected component and when the configuration is restricted to rectangles, parallelograms, triangles and finally obtain a lower bound for the optimal case of one connected component.

\subsection{A Single Rectangle}
Let the rectangle for white queens be $[\,[0,0]\,, \,[a,0]\,, \,[a,b]\,,\, [0,b]\,]$, with the obvious fact that for a rectangle with a given size, placing it in the corner will lead to the largest non-attacking area. The area for white queens is $ab$ and the area for black queens is $(1-a-b)^2$. We'd like to find the maximum of $ab$ under the condition 
$$
ab = (\max(1-a-b, 0))^2, \quad 0 \leq a,b \leq 1.
$$

\begin{figure}[h!]
  \center
  \includegraphics[width=0.75\textwidth]{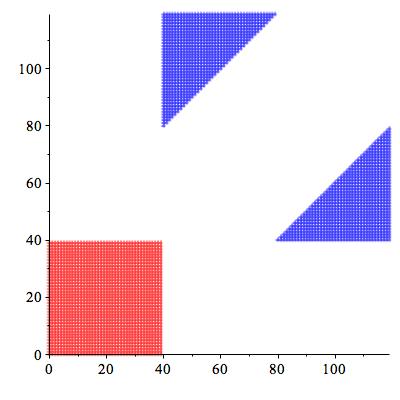}
  \caption{The Optimal Rectangle for a 120 by 120 Chess Board}
\end{figure}

Since $a$ and $b$ are symmetric, the maximum must be on the line $a=b$. Hence the optimal case is when 
$$
a=b=\frac{1}{3}
$$
and the largest area for peaceable queens when the configuration for white queens is a rectangle is $\frac{1}{9}$.

\subsection{A Single Parallelogram}
Let the parallelogram for white queens be $[\,[0,0]\,,\,[a,a]\,,\,[a,a+b]\,,\,[0,b]\,]$.

\begin{figure}[h!]
  \center
  \includegraphics[width=0.75\textwidth]{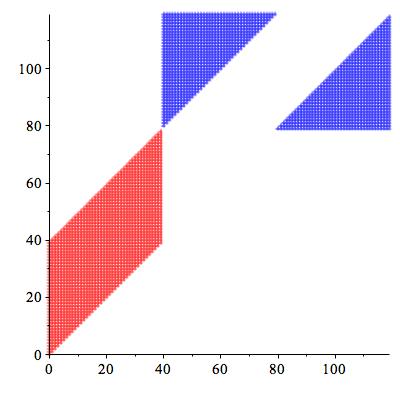}
  \caption{The Optimal Parallelogram for a 120 by 120 Chess Board}
\end{figure}

Note that as mentioned in the beginning of this section, because the line segment must be vertical, horizontal or of slope $\pm 1$ and the corner is the best place to locate a shape, there are only two kinds of parallelograms, the other one being $[[0,0],[b,0],[a+b,a],[a,a]]$. Obviously they are symmetric with respect to the line $y=x$, so let's focus on one of them. 

The area for white queens is still $ab$ and the area for black queens is still $(\max(1-a-b, 0))^2$. So similarly with the rectangle case, the maximum area $\frac{1}{9}$ is reached when
$$
a=b=\frac{1}{3}.
$$

\subsection{A Single Triangle}
With similar arguments as in the last subsection, the optimal triangle must have the format: $[\,[0,0]\,,\,[0,a]\,,\,[a,a]\,]$. The area for white queens is 
$$
\frac{1}{2} a^2
$$
and the area for black queens is 
$$
\frac{1}{2} (1-a)^2.
$$
with the condition $0 \leq a \leq 1$.

Hence, when $a=\frac{1}{2}$ the area reaches its maximum $\frac{1}{8}$, which is better than the rectangle or parallelogram configuration. 

\begin{figure}[h!]
  \center
  \includegraphics[width=0.75\textwidth]{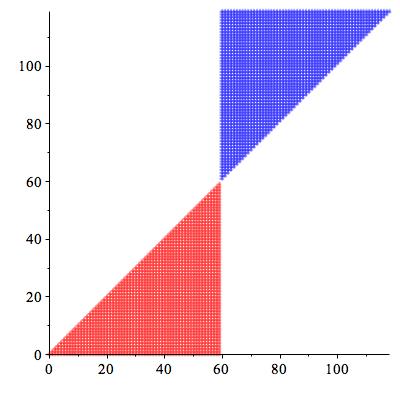}
  \caption{The Optimal Triangle for a 120 by 120 Chess Board}
\end{figure}

By the way, $[\,[0,0]\,,\,[0,a]\,,\,[a,0]\,]$ won't be a good candidate for optimal triangles because we can always extend it to a square $[\,[0,0]\,,\,[0,a]\,,\, [a,a] \, , \, [a,0]\,]$ without decreasing the area of black queens. Then its maximum cannot exceed the maximum of rectangles, which is $\frac{1}{9}$.

\subsection{A Single Hexagon}
After looking at specific configurations in the above three subsections, we'd like to find some numerical lower bounds for the single connected component configuration. It is interesting to find out or at least get a numerical estimation how large the area of white or black queens can be if the white queens are in a single connected component.  Note that from rectangles and parallelogram we get a lower bound $\frac{1}{9} \approx 0.1111$ and from triangles we get a better lower bound $\frac{1}{8} = 0.125$.

The natural thing is that we want to place the polygon in a corner. Because of the restriction of the orientations of its sides, at most it can be an octagon. Let's place the polygon in the lower left corner. Then we immediately realize that it is a waste if the polygon doesn't fill the lower left corner of the unit square. It is the same for the upper right side of the polygon. If part of its vertices are $[\,[a,b]\,,\, [a, b+c]\,,\,[a-d, b+c+d]\,,\,[a-d-f, b+c+d]\,]$, then we can always extend the polygon to $[\, \dots \,,\,[a,b]\,,\, [a, b+c+d]\,,\, [a-d-f, b+c+d]\,, \, \dots \, ] $ without decreasing the area of black queens. 

Hence the general shape is a hexagon 
$$
[\, [0,0]\,,\,[a,0]\,,\,[a+b,b]\,,\,[a+b, b+c]\,,\,[d, b+c]\,,\,[0,b+c-d]\,]
$$ 
with four parameters. Then the area for white queens is 
$$
(a+b)(b+c) - \frac{1}{2} (b^2+d^2),
$$
and the area for black queens is 
$$
\frac{1}{2} (1-a-b-c)^2 + \frac{1}{2} (1-a-2b-c+d)^2.
$$
With the procedure {\tt MaxC}, one of the local maximums found using Lagrange multipliers is when 
$$
a=c=d = \frac{1}{2}, \, b=0.
$$
However, actually this is the optimal triangle with an area of $\frac{1}{8}$.

\begin{figure}[h!]
  \center
  \includegraphics[width=0.75\textwidth]{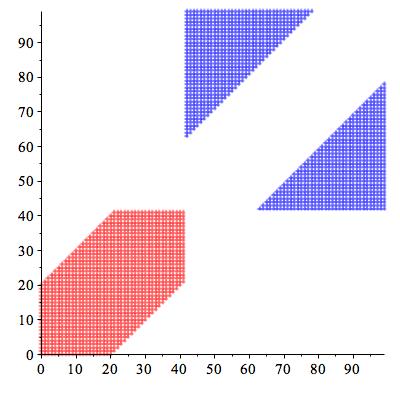}
  \caption{The Nearly Best Lower Bound Configuration for a 100 by 100 Chess Board}
\end{figure}

Another local maximum is when $a=b=c=d$. In that case, we have 
$$
3a^2 = (1-3a)^2.
$$
Hence when 
$$a = \frac{3-\sqrt{3}}{6} \approx 0.2113248654,$$ the area of white queens is maximized at 
$$ 3a^2 = \frac{2-\sqrt{3}}{2} \approx 0.1339745962. $$
The best configuration of hexagons is found and at least we have a numerical lower bound 0.1339745962 for the best single component configuration.

\section{Two Components}
Since in Jubin's construction, there are two pentagons, it is natural to think of the optimum of certain two-component configurations. The difficulty for analyzing the two-component is that more parameters are introduced and the area formula for black queens becomes a much more complicated piece-wise function. 

In this section, the cylindrical algebraic decomposition algorithm in quantifier elimination is applied to find out the exact optimal parameters and the maximum areas. Given a set $S$ of polynomials in $\RR^n$, a cylindrical algebraic decomposition is a decomposition of $\RR^n$ into semi-algebraic connected sets called cells, on which each polynomial has constant sign, either +, - or 0. With such a decomposition it is easy to give a solution of a system of inequalities and equations defined by the polynomials, i.e. a real polynomial system.

\subsection{Two Identical Squares}
To keep the number of parameters as few as possible, the configuration of two identical squares is the first we'd like to study. There are two parameters, the side length $a$ and the $x$-coordinate $s$ of the lower left vertex of the right square, the left square's lower left vertex being the origin. 

The two squares are 
$$[\, [0,0] \,,\, [a,0] \,,\, [a,a] \,,\, [0,a] \, ]$$
and 
$$[\, [s,0] \,,\, [s+a,0] \,,\, [s+a,a] \,,\, [s,a] \, ].$$ 

\begin{figure}[h!]
  \center
  \includegraphics[width=0.75\textwidth]{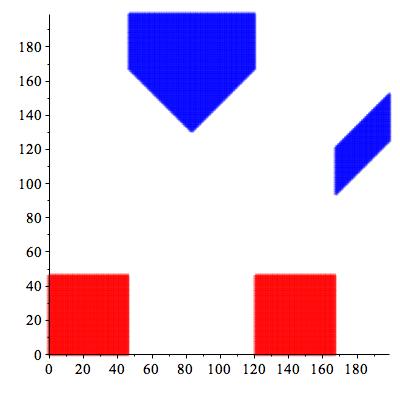}
  \caption{The Nearly Optimal Two Identical Squares Configuration for a 200 by 200 Chess Board}
\end{figure}

Based on this configuration, the domain is 
$$
0 \leq a \leq \frac{1}{2}, \quad a \leq s \leq 1-a.
$$

The area of white queens is 
$$
2a^2.
$$

Actually the formula for black queens is very complicated, especially when $a$ is small there may be a lot of components for $B$. However, by experimentation (procedure {\tt FindM2Square}), we found that for all mid-range $s \in [0.24, 0.76]$, $a$ around 0.23 will always maximize the area. Then we just need to focus on the shape of $B$ when $a$ is not far from its optimum. 

The area of black queens is 
$$
(s-a)(1-s-a) + \frac{1}{4}(s-a)^2 + (\max(1-s-2a, 0))^2 + \max(s-2a,0)(1-s-a).
$$

The domain for $a$ and $s$ is a triangle. The area formula for black queens shows that the two lines $s=2a$ and $s=1-2a$ separate the domain into 4 regions. In each region, we have a polynomial formula for the area of black queens. Since the area of white queens $W$ is just a simple formula of $a$, we need to maximize $a$ with the condition $W=B$.

When $s \geq 2a$ and $s \geq 1-2a$, by cylindrical algebraic decomposition we obtained 
$$
\begin{cases} 
    \frac{1}{2}(-1+\sqrt{2}) \leq a < \frac{1}{27}(1+2\sqrt{7}) & s=\frac{4+a}{7} + \frac{2}{7} \sqrt{4-19a+9a^2} \\
    \frac{1}{27}(1+2\sqrt{7}) \leq a < \frac{1}{18}(19-\sqrt{217}) &s=\frac{4+a}{7} \pm \frac{2}{7} \sqrt{4-19a+9a^2} \\
    a=\frac{1}{18}(19-\sqrt{217}) & s= \frac{4+a}{7} - \frac{2}{7} \sqrt{4-19a+9a^2}
\end{cases}.
$$

When $s \leq 2a$ and $s \geq 1-2a$, the result is an empty set. 

When $s \leq 2a$ and $s \leq 1-2a$, we obtained
$$
\frac{2}{9} \leq a \leq \frac{1}{7}(3-\sqrt{2}), \quad s=2-7a-2\sqrt{-2a+9a^2}.
$$

When $s \geq 2a$ and $s \leq 1-2a$, we obtained
$$
\frac{2}{9} \leq a \leq \frac{1}{27}(1+2\sqrt{7}), \quad s=3a - \frac{2}{\sqrt{3}} \sqrt{1-7a+12a^2}.
$$

Comparing the four cases, we found that the largest area occurred in case 1, when 
$$
a = \frac{1}{18} (19-\sqrt{217}) \approx 0.2371711193,
$$
$$
s = \frac{13}{18} - \frac{1}{126} \sqrt{217} \approx 0.6053101598 .
$$
The largest area is ${\frac {289}{81}}-{\frac {19\,\sqrt {217}}{81}} \approx  0.112500281.$

\subsection{Two Identical Triangles}
The configuration of two identical isosceles right triangles with the same orientation is the next to be considered. There are also two parameters, the leg length $a$ and the $x$-coordinate $s$ of the lower left vertex of the triangle on the right. Note that the slopes of both triangles' hypotenuses are $+1$. 

The two isosceles right triangles are
$$[\,[0,0]\,,\,[a,0]\,,\,[a,a]\,]$$
and 
$$
[\,[s,0]\,,\,[a+s,0]\,,\,[a+s,a]\,].
$$
The domain for the two parameters $a$ and $s$ is also 
$$
0 \leq a \leq \frac{1}{2}, \quad a \leq s \leq 1-a.
$$
The area of white queens is $a^2$ and for the area of black queens, by numerical experimentation, we found that for all mid-range $s \in [0.32, 0.68]$, the area is maximized when $a$ is around 0.31. Hence for $a$ around 0.31, we have that the area of black queens is 
$$
2(s-a)(1-s-a) + \frac{1}{2} (s-a)^2 + \frac{1}{2} (1-s-a)^2 + \frac{1}{2} (\max (1-s-2a, 0))^2.
$$
When $s \geq 1-2a$, by cylindrical algebraic decomposition we obtained 
$$
\begin{cases} 
    \frac{1}{2}(2-\sqrt{2}) \leq a < \frac{1}{4}(-1+\sqrt{5}) & s=\frac{1}{2} + \frac{1}{2} \sqrt{3-12a+8a^2} \\
   \frac{1}{4}(-1+\sqrt{5}) \leq a < \frac{1}{4}(3-\sqrt{3}) &s=\frac{1}{2} \pm \frac{1}{2} \sqrt{3-12a+8a^2} \\
    a=\frac{1}{4}(3-\sqrt{3}) & s=\frac{1}{2} - \frac{1}{2} \sqrt{3-12a+8a^2}
\end{cases}.
$$
When $s \leq 1-2a$, we obtained 
$$
\frac{1}{11}(5-\sqrt{3}) \leq a \leq \frac{1}{4}(-1+\sqrt{5}), s=2a - \sqrt{2-10a+12a^2}.
$$
Hence the area is maximized when
$$a = \frac{1}{4}(3-\sqrt{3}) \approx 0.316987298,$$
$$
s = \frac{1}{2}.
$$
The largest area is $\frac{3}{4} - \frac{3}{8} \sqrt{3} \approx 0.1004809470$.

\begin{figure}[h!]
  \center
  \includegraphics[width=0.75\textwidth]{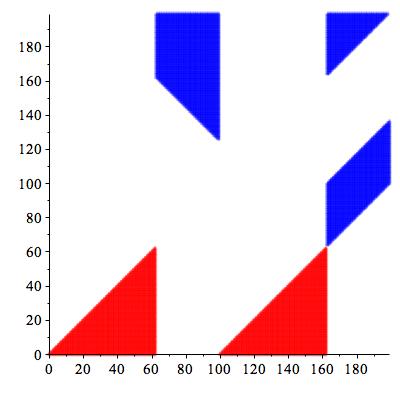}
  \caption{The Nearly Optimal Two Identical Isosceles Right Triangles with the Same Orientation Configuration for a 200 by 200 Chess Board}
\end{figure}

Actually, a larger area can be obtained if two identical isosceles right triangles with different orientations are considered. For example, if we take the two triangles to be $$[\,[0,0]\,,\,[a,0]\,,\,[a,a]\,]$$
and 
$$
[\,[1-a,0]\,,\,[1,0]\,,\,[1-a,a]\,],
$$
then the area of black queens is 
$$
a(1-2a) + (\frac{1}{2} - a)^2 = -a^2 + \frac{1}{4}.
$$
Equalizing the areas of white queens and black queens, we get 
$$
\textrm{Area}(W) = a^2 = \frac{1}{8},
$$
which is greater than the optimal case of two identical isosceles right triangles with the same orientation.

\begin{figure}[h!]
  \center
  \includegraphics[width=0.75\textwidth]{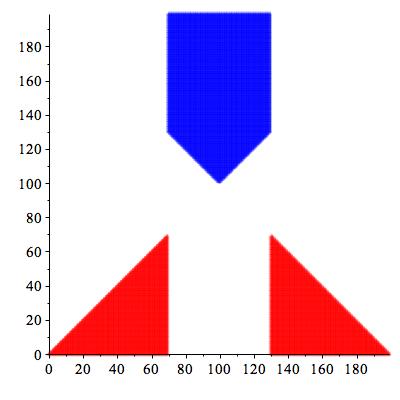}
  \caption{An Example of Two Identical Isosceles Right Triangles with Different Orientations Configuration for a 200 by 200 Chess Board}
\end{figure}

\subsection{One Square and One Triangle with the Same Side Length}

\begin{figure}[h!]
  \center
  \includegraphics[width=0.75\textwidth]{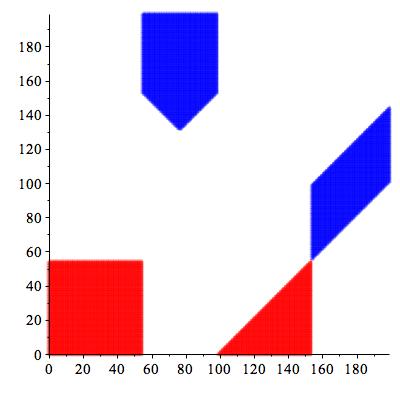}
  \caption{The Nearly Optimal One Square and One Triangle (with the same side length) Configuration for a 200 by 200 Chess Board}
\end{figure}

With the same notations as the above two subsections, let $W$ be the union of the square 
$$[\, [0,0] \,,\, [a,0] \,,\, [a,a] \,,\, [0,a] \, ]$$
and the triangle
$$ [\,[s,0]\,,\,[a+s,0]\,,\,[a+s,a]\,]. $$
Then the area of white queens is $\frac{3}{2} a^2$ and the area of black queens is 
$$
a(s-a)(1-s-a) + \frac{1}{4} (s-a)^2 + (\max(1-s-2a, 0))^2
$$
when $a$ is around its optimum 0.27 and $s \in [0.28, 0.72]$. It is obtained that when $s \geq 1-2a$
$$
\begin{cases} 
    \frac{1}{2}(-2+\sqrt{6}) \leq a < \frac{1}{21}(1+\sqrt{22}) & s=\frac{4-a}{7} + \frac{1}{7} \sqrt{16-64a+22a^2} \\
   \frac{1}{21}(1+\sqrt{22}) \leq a < \frac{2}{11}(8-\sqrt{42}) &s=\frac{4-a}{7} \pm \frac{1}{7} \sqrt{16-64a+22a^2} \\
    a=\frac{2}{11}(8-\sqrt{42}) & s=\frac{4-a}{7} - \frac{1}{7} \sqrt{16-64a+22a^2}
\end{cases},
$$
and when $s \leq 1-2a$
$$
\frac{1}{15} (6-\sqrt{6}) \leq a \leq \frac{1}{21} (1+\sqrt{22}), \quad s = \frac{7a}{3} - \frac{1}{3} \sqrt{12-72a+106a^2}.
$$
Consequently, we have the maximized area when
$$
a = \frac{2}{11} (8- \sqrt{42} ) \approx 0.276228965,
$$
$$
s = \frac{112}{33} - \frac{14}{33} \sqrt{42} - \frac{50}{33} \sqrt{7} + \frac{52}{33} \sqrt{6} \approx 0.495622162.
$$
The largest area is $\frac{636}{121} - \frac{96}{121} \sqrt{42} \approx 0.1144536616$. Among the three configurations where CAD is applied in this section, we found that this configuration with one square and one triangle has the largest area. 

\section{Remarks}
Our method can be easily generalized for configurations with more components and/or more parameters. For instance, let's consider the configuration of two squares, not necessarily identical. Then there are three parameters, the side length $a$ of the left square, the side length $b$ of the right square and the $x$-coordinate $s$ of the right square's lower left vertex. For fixed $b$ and $s$, we can find the interval of $a$ in which the optimum is located. Then for each fixed $s$, we are able to find the interval of $b$ such that its corresponding $a$ will lead to the largest area $a^2 + b^2$. When the estimated optimal parameters are determined, a piece-wise function of the area of black queens follows. 

The main difficulty of this peaceable queens problem lies in the number of parameters and the complexity of the area formula of black queens. When there are multiple components, as long as the number of parameters is limited, it should be still doable. For example, the configuration of three identical squares which are placed equidistantly has only one parameter, the side length $a$. When the chess board is 240 by 240, the optimal $a$ is around 40, which means in the unit square the optimal side length is around $\frac{1}{6}.$ 

In conclusion, in this chapter we prove that Jubin's configuration is a local optimum. Optimal cases of some certain configurations are discussed. Future work includes the exact solution of complicated configurations with numerous parameters, whether the white queens have two components under the best configuration, and proof or disproof that Jubin's configuration is indeed the best.

%\bibliographystyle{amsalpha}
%\bibliography{biblio}

%NOTE: Older instructions ask for a CV of some kind at the end. This is no longer required.

\end{document}